\documentclass{amsart}
\usepackage{amssymb}
\usepackage[utf8]{inputenc}
\usepackage{amsmath,amsthm,amstext,amssymb,bbm}
\usepackage{amscd,  amsfonts,  amsbsy,  mathrsfs, upgreek, mathtools, stmaryrd, bbm}
\usepackage{pdfsync}

\usepackage{geometry}
\newtheorem{theorem}{Theorem}[section]
\newtheorem{proposition}[theorem]{Proposition}
\newtheorem{definition}[theorem]{Definition}
\newtheorem{lemma}[theorem]{Lemma}
\newtheorem{corollary}[theorem]{Corollary}
\newtheorem{question}[theorem]{Question}

\newtheorem*{claim*}{Claim}
\newtheorem*{subclaim*}{Subclaim}

\newcommand{\Ce}{{\mathcal{C}}}
\newcommand{\De}{{\mathcal{D}}}

\newcommand{\ESR}{\mathrm{ESR}}
\newcommand{\SR}{\mathrm{SR}}

\newcommand{\otp}[1]{{{\rm{otp}}\left(#1\right)}}

\newcommand{\seq}[2]{\langle{#1}~\vert~{#2}\rangle}

\newcommand{\ran}[1]{{{\rm{ran}}(#1)}}
\newcommand{\id}{{\rm{id}}}
\newcommand{\crit}[1]{{{\rm{crit}}({#1})}}
\newcommand{\calL}{\mathcal{L}}
\newcommand{\rank}[1]{{\rm{rnk}}({#1})}
\newcommand{\cof}[1]{{{\rm{cof}}(#1)}}
\newcommand{\Set}[2]{\{{#1}~ \vert~{#2}\}}

\newcommand{\LL}{{\rm{L}}}

\newcommand{\On}{{\rm{Ord}}}

\newcommand{\betrag}[1]{\vert{#1}\vert}

\renewcommand{\emptyset}{\varnothing}

\title[Huge Reflection]{Huge Reflection}
\author{Joan Bagaria}
\address{ICREA (Instituci\'o Catalana de Recerca i Estudis Avan\c{c}ats) and
\newline \indent Departament de Matem\`atiques i Inform\`atica, Universitat de Barcelona. 
Gran Via de les Corts Catalanes, 585,
08007 Barcelona, Catalonia.}
\email{joan.bagaria@icrea.cat}

\author{Philipp L\"{u}cke}
\address{Fachbereich Mathematik, Universit\"at Hamburg, Bundesstra{\ss}e 55, Hamburg, 20146, Germany}
\email{philipp.luecke@uni-hamburg.de}

\thanks{This is an update of the published version of the paper (\emph{Ann. Pure Appl. Logic} 174, No.1, Paper No. 103171, 2023) that corrects a problem in the definition of \emph{weakly exact cardinals} (see Definition \ref{definition:WeaklyExactCardinalUp} below). This problem was brought to our attention by Nai-Chung Hou, a student of the first author. It can be fixed by replacing the notion of \emph{$\Pi_n(V_\kappa)$-correctness} with the weaker concept of \emph{$\Pi_n(V_\kappa)$-upwards correctness}. Through this modification, the proofs of all of the main theorems of the paper hold as before, and only three minor results (Propositions 5.6 and 7.2, and Corollary 6.5 of the published version) cannot be saved. 
The authors would like to thank Nai-Chung Hou for pointing out this problem. 
Part of this research was  supported by  the Generalitat de Catalunya (Catalan Government) under grant SGR 270-2017, and by the Spanish
Government under grant MTM2017-86777-P. 
 In addition, this project has received funding from the European Union’s Horizon 2020 research and innovation programme under the Marie Sk{\l}odowska-Curie grant agreement No 842082 (Project \emph{SAIFIA: Strong Axioms of Infinity -- Frameworks, Interactions and Applications}). 
 }

\date{\today }

\begin{document}

\begin{abstract}
We study \emph{Structural Reflection} beyond \emph{Vop\v{e}nka's Principle}, at the level of almost-huge cardinals and higher, up to rank-into-rank embeddings. We identify and classify new large cardinal notions in that region that correspond to some form of what we call \emph{Exact Structural Reflection} ($\ESR$). Namely, given cardinals  $\kappa<\lambda$ and a class $\Ce$ of structures of the same type, the  
corresponding instance of $\ESR$ asserts that for every structure $A$ in $\Ce$ of rank $\lambda$, there is a structure $B$ in $\Ce$ of rank $\kappa$ and an elementary embedding of $B$ into $A$.  Inspired by the statement of \emph{Chang's Conjecture}, we also introduce and study   sequential forms of $\ESR$, which, in the case of sequences of length $\omega$, turn out to be very strong. Indeed, when  restricted to $\Pi_1$-definable  classes of structures they follow from the existence of $I1$-embeddings, while for more complicated classes of structures, e.g., $\Sigma_2$, they are not known to be consistent. Thus, these principles unveil a new class of large cardinals that go beyond $I1$-embeddings, yet they may not fall into \emph{Kunen's Inconsistency}. 
\end{abstract}
\maketitle


\section{Introduction}

Given a class\footnote{We work in ZFC. So, proper classes are always definable, with set parameters.} $\Ce$ of structures\footnote{Throughout this paper, the term \emph{structure}  refers to structures for countable first-order languages. 
 The \emph{cardinality} (respectively, the \emph{rank}) of a structure is defined as the cardinality (respectively, the rank) of its domain.}  of the same type and a cardinal $\kappa$, the principle of \emph{Structural Reflection\footnote{A detailed discussion of this principle can be found in \cite{BagariaRefl}.}  $\SR$}  holds at $\kappa$ for $\Ce$ if for every structure $A$ in $\Ce$, there exists some $B\in \Ce \cap V_\kappa$ and an elementary embedding of  $B$ into $A$. 
 Different forms of $\SR$ have been investigated in \cite{Ba:CC, BagariaRefl, BCMR, BV, BagariaWilsonRefl, Lu:SR}, yielding canonical characterizations  of large cardinals in different regions of the large cardinal hierarchy. 
 For example,  results in  \cite{Ba:CC} and \cite{BCMR} use Magidor's classical characterization of supercompact cardinals from \cite{Mag} to show that the existence of such a  cardinal is equivalent to the validity of the principle $\SR$ for all classes of structures definable by $\Pi_1$-formulas without parameters.

 The principles of structural reflection considered so far correspond to large cardinals up to \emph{Vop\v{e}nka's Principle}, stating that  every proper class of structures of the same type contains a structure that is elementary embeddable into another structure in the given class. 
 The validity of this principle can be shown to be equivalent to the existence of cardinals witnessing $\SR$ for every class of structures (see \cite{Ba:CC}). 
 In this paper, we shall study principles of structural reflection that correspond to large cardinal notions  stronger than  Vop\v{e}nka's Principle,  up to rank-into-rank embeddings, and beyond. These principles are given by variations of  the following \emph{exact} form of $\SR$:

\begin{definition}[Exact Structural Reflection] \label{definition:ESR}
 Given infinite cardinals $\kappa<\lambda$ and a class $\Ce$ of structures of the same type, we let $\ESR_\Ce(\kappa,\lambda)$ denote the assertion that for every $A\in \Ce$ of rank $\lambda$, there exists some $B\in \Ce$ of rank $\kappa$ and an elementary embedding from $B$ into $A$. 
\end{definition}

Given a definability class $\Gamma$ (i.e., one of $\Sigma_n$ or $\Pi_n$, for some $n<\omega$) and a class $P$, we  introduce the following  variations of the above definition that will allow us to formulate our results in a compact way:  
 \begin{enumerate}
  \item We let $\Gamma(P)$-$\ESR(\kappa,\lambda)$ denote the statement that $\ESR_\Ce (\kappa,\lambda)$ holds for every  class $\Ce$ of structures of the same type that is $\Gamma$-definable with parameters in $P$. 
  
  \item We let $\Gamma(P)$-$\ESR(\kappa)$ denote the statement that $\Gamma(P)$-$\ESR(\kappa,\lambda)$ holds for some cardinal $\lambda >\kappa$. 
       
 
  \item  We let $\Gamma (P)^{ic}$-$\ESR(\kappa,\lambda)$ and $\Gamma (P)^{ic}$-$\ESR(\kappa)$ 
  denote  the restrictions of the respective  principles to classes of structures that are $\Gamma$-definable with parameters in $P$ \emph{and} are closed under isomorphic copies.  
 \end{enumerate}

Using the \emph{Downward L\"owenheim--Skolem Theorem}, 
it is easy to see that $\ESR_\Ce(\kappa,\lambda)$ holds for every countable first-order language $\calL$, every elementary class $\Ce$ of $\calL$-structures and all uncountable cardinals $\kappa<\lambda$ with $\cof{\kappa}\leq\cof{\lambda}$ (see Proposition \ref{proposition:CharClosedIso} below). 
 In contrast, we shall see that the above  principles for  externally defined classes are quite strong, for they correspond to large cardinals in the region between supercompact and rank-into-rank embeddings, and beyond. Below is a summary of the main results.

 First, we discuss our results for classes of structures closed under isomorphic copies. While an easy application of $\Sigma_1$-absoluteness shows that the principle  $\Sigma_1(V_\kappa)^{ic}$-$\ESR(\kappa,\lambda)$ holds for all uncountable cardinals $\kappa<\lambda$ with $\cof{\kappa}\leq\cof{\lambda}$ (Proposition \ref{proposition:Sigma1ClosedTrivial}), we  prove that the principle $\Pi_1^{ic}$-$\ESR(\kappa,\lambda)$ already implies the existence of a ${<}\lambda$-supercompact cardinal less than or equal to $\kappa$ (Lemma \ref{lemma:P1ESRisoSupercompacts}). 
 Moreover, for singular cardinals $\kappa$, our results show that the validity of principles of the form $\Pi_1(V_\kappa)^{ic}$-$\ESR(\kappa,\lambda)$ is equivalent to the existence of cardinals below $\kappa$ possessing certain degrees of supercompactness. In particular, it turns out that singular limits of supercompact cardinals can be characterized through exact structural reflection for $\Pi_1$-definable classes closed under isomorphic copies. Namely, we have the following equivalences:

\begin{theorem}\label{theorem:LimitSupercompacts}
 The following statements are equivalent for every singular cardinal $\kappa$: 
 \begin{enumerate}
     \item $\kappa$ is a  limit of supercompact cardinals.  
     
     \item $\Pi_1(\kappa)^{ic}$-$\ESR(\kappa,\lambda)$ holds for a proper class of cardinals $\lambda$.  
     
     \item $\Sigma_2(V_\kappa)^{ic}$-$\ESR(\kappa,\lambda)$ holds for a proper class of cardinals $\lambda$. 
 \end{enumerate}
\end{theorem}

 In order to state an analogous result  for more complicated classes of structures, we have to introduce a weak form of the notion of  \emph{$C^{(n)}$-extendibility} from \cite{Ba:CC}. 
 Recall that a cardinal $\kappa$ is \emph{$\lambda$-extendible} for some ordinal $\lambda>\kappa$ if there is an ordinal $\eta$ and an elementary embedding $j:V_\lambda\to V_\eta$ with $\crit{j}=\kappa$ and $j(\kappa)>\lambda$. 
 Following \cite{Ba:CC}, for every  $n<\omega$, we let $C^{(n)}$ denote the $\Pi_n$-definable closed unbounded class of all of ordinals $\alpha$ such that $V_\alpha$ is a $\Sigma_n$-elementary substructure of $V$. 
 %
 Given cardinals $\kappa<\lambda$ and $n<\omega$, the cardinal $\kappa$ is \emph{$\lambda$-$C^{(n)}$-extendible} if there is an elementary embedding $j:V_{\lambda}\to V_{\nu}$ for some cardinal  $\nu$ with  $\crit{j}=\kappa$,  $j(\kappa)>\lambda$ and $j(\kappa)\in C^{(n)}$. 
 In addition, we say that $\kappa$ is  \emph{$C^{(n)}$-extendible} if it is $\lambda$-$C^{(n)}$-extendible for all (equivalently, for a proper class of) $\lambda >\kappa$ (see {\cite[Section 3]{Ba:CC}}).
 

The following weaker form of $C^{(n)}$-extendibility will allow us to  prove a version of Theorem \ref{theorem:LimitSupercompacts} for  classes of structures of complexity greater than $\Sigma_2$.

\begin{definition}
Given ordinals $\mu<\lambda$ and a natural number $n$, a cardinal $\kappa\leq\mu$ is \emph{$[\mu,\lambda)$-$C^{(n)}$-extendible} if there exist $\nu\in C^{(n+1)}\cap[\mu,\lambda)$ and an elementary embedding $j:V_\lambda\to V_\eta$, for some $\eta$,  with $\crit{j}=\kappa$, $j(\mu)\geq\lambda$ and $j(\nu)\in C^{(n)}$. In addition, we say that $\kappa$ is $[\mu, \infty)$-$C^{(n)}$-extendible if it is $[\mu ,\lambda)$-$C^{(n)}$-extendible for a proper class of ordinals $\lambda$.
\end{definition}

It is easy to see that extendible cardinals $\kappa$ are $[\mu,\infty)$-$C^{(1)}$-extendible for all $\mu\geq\kappa$. Using the fact that the requirement "$j(\kappa)>\lambda$" can be omitted in the definition of extendibility (see {\cite[Proposition 23.15]{Kan:THI}}), we can also see that a cardinal $\kappa$ is extendible if and only if it is $[\mu,\infty)$-$C^{(1)}$-extendible for some $\mu\geq\kappa$. 
We will later show  that every $C^{(n)}$-extendible cardinal is $[\mu ,\infty)$-$C^{(n)}$-extendible, for every $\mu\geq \kappa$  (Proposition \ref{proposition:CnExtendibleIntervalExtendible}). 

Using this notion, the above characterization of singular limits of supercompact cardinals now generalizes in the following way:

\begin{theorem}\label{theorem:LimitCnExtendibles}
 For every $n>0$, the following statements are equivalent for every singular cardinal $\kappa$: 
 \begin{enumerate}
     \item $\kappa$ is a  limit of $[\kappa,\infty)$-$C^{(n)}$-extendible  cardinals. 
     
     \item $\Pi_{n+1}(\kappa)^{ic}$-$\ESR(\kappa,\lambda)$ holds for a proper class of cardinals $\lambda$. 
     
     \item $\Sigma_{n+2}(V_\kappa)^{ic}$-$\ESR(\kappa,\lambda)$ holds for a proper class of cardinals $\lambda$. 
 \end{enumerate}
\end{theorem}

 In particular, Theorem \ref{theorem:LimitCnExtendibles} shows that a singular cardinal $\kappa$ is a limit of extendible cardinals if and only if $\Pi_2(\kappa)^{ic}$-$\ESR(\kappa,\lambda)$ holds for a proper class of cardinals $\lambda$.

 In combination with results from \cite{Ba:CC} and \cite{BCMR}, the methods developed in the proof of Theorem \ref{theorem:LimitCnExtendibles} also allow us to conclude that exact structural reflection for classes of structures closed under isomorphisms holding at singular cardinals does not imply the existence of large cardinals stronger than Vop\v{e}nka's Principle. 
 In fact,  Vop\v{e}nka's Principle can be characterized through the validity of principles of the form $\Pi_n(V_\kappa)^{ic}$-$\ESR(\kappa,\lambda)$.

\begin{theorem}\label{theorem:ESR-VP}
 Over the theory {\rm{ZFC}}, the following schemes of sentences imply each other: 
 \begin{enumerate}
     \item Vop\v{e}nka's Principle. 
     
     \item For every class $\Ce$  of structures of the same type that is closed under isomorphic images,  there is a  cardinal $\kappa$ with  the property that $\ESR_\Ce(\kappa,\lambda)$ holds for all  $\lambda>\kappa$. 
     
     \item For every natural number $n>0$, there exists a proper class of cardinals $\kappa$ with the property that $\Pi_n(V_\kappa)^{ic}$-$\ESR(\kappa,\lambda)$ holds for all $\lambda>\kappa$. 
 \end{enumerate}
\end{theorem}

In contrast to the above results, both the validity of the principle $\Pi_1^{ic}$-$\ESR$ at a regular cardinal and the validity of the principle $\Pi_1$-$\ESR$ at some cardinal  turn out to imply the existence of  large cardinals stronger than Vop\v{e}nka's Principle, e.g. almost huge cardinals. 
 The large cardinal properties introduced below will allow us to  capture the  strength of these  forms of exact structural reflection. Their definition is motivated by results in \cite{Lu:SR} that provide a characterization of \emph{shrewd cardinals} (introduced by Rathjen in \cite{MR1369172}) through a variation of Magidor's classical characterization of supercompactness in \cite{Mag} and  similar characterizations of $n$-hugeness in  {\cite[Section 6]{MR3913154}}. 
 In the following, we will say that a set  $M$ is \emph{$\Pi_n(P)$-upwards correct} for some natural number $n>0$ and a class $P$ if all $\Pi_n$-formulas with parameters in $M\cap P$ that hold in $M$ also hold in $V$.

\begin{definition}\label{definition:WeaklyExactCardinalUp}
 Given  a natural number $n>0$, an infinite cardinal $\kappa$ is \emph{weakly $n$-exact for a cardinal $\lambda>\kappa$} if for every $A\in V_{\lambda +1}$, there exists 
     a transitive, $\Pi_n(V_{\kappa+1})$-upwards correct set $M$ with $V_\kappa\cup \{\kappa\} \subseteq M$,  a cardinal $\lambda'\in C^{(n-1)}$ greater than $\beth_\lambda$   
     and an elementary embedding $j:M\to H_{\lambda'}$ with $j(\kappa)=\lambda$ and $A\in\ran{j}$.  
     If we further require that $j(\crit{j})=\kappa$, then we say that $\kappa$ is \emph{weakly parametrically $n$-exact for $\lambda$}.
 
\end{definition}

Observe  that if $\kappa$ is weakly parametrically $1$-exact for $\lambda$, then $\kappa$ and $\lambda$ are both inaccessible. The following result shows how weakly $n$-exact cardinals are connected to principles of exact structural reflection for $\Pi_n$-definable classes of structures.

\begin{theorem}\label{theorem:CharacterizationPI-ESR}
 The following statements are equivalent for all cardinals $\kappa$ and all natural numbers $n>0$: 
 \begin{enumerate}
     \item $\kappa$ is the least regular cardinal such that $\Pi_n^{ic}$-$\ESR(\kappa)$ holds. 
     
     \item $\kappa$ is the least  cardinal such that $\Pi_n$-$\ESR(\kappa)$ holds. 
     
     \item $\kappa$ is the least cardinal such that $\Pi_n(V_\kappa)$-$\ESR(\kappa)$ holds. 
     
     \item $\kappa$ is the least 
      cardinal that is weakly $n$-exact for some  $\lambda>\kappa$. 
     
     \item $\kappa$ is the least cardinal that is weakly parametrically  $n$-exact for some  $\lambda>\kappa$. 
 \end{enumerate}
\end{theorem}

 In the case of $\Sigma_{n+1}$-definable classes of structures, the   large cardinal principles corresponding to the different forms of exact structural reflection are the following:

\begin{definition}
\label{defnexact}
 Given a natural number $n$, an infinite  cardinal $\kappa$ is \emph{$n$-exact for some cardinal $\lambda>\kappa$} if for every $A\in V_{\lambda +1}$, there exists a cardinal $\kappa'\in C^{(n)}$ greater than $\beth_\kappa$, a cardinal  $\lambda'\in C^{(n+1)}$ greater than $\lambda$, an elementary submodel $X$ of $H_{\kappa'}$ with  $V_\kappa \cup \{\kappa\}   \subseteq X$, and an elementary embedding $j:X\to H_{\lambda'}$ with  $j(\kappa)=\lambda$ and   $A\in\ran{j}$. 
     If we further require that   $j(\crit{j})=\kappa$ holds,\footnote{Even though $X$ need not be transitive, we still define $\crit{j}$ as the least ordinal moved by $j$, which exists since $j$ is not the identity on the ordinals as $j(\kappa)=\lambda$.} then we say that $\kappa$ is \emph{parametrically $n$-exact for $\lambda$}.
 
\end{definition}

Note that, if $m\leq n<\omega$ and $\kappa$ is (parametrically) $n$-exact for $\lambda$, then $\kappa$ is also (parametrically) $m$-exact for $\lambda$. 
Moreover, standard arguments show that, if $\kappa$ is parametrically $0$-exact for $\lambda$, then both $\kappa$ and $\lambda$ are inaccessible cardinals.
 In addition, it is easily seen that, given $0<n<\omega$, if  $\kappa$ is (parametrically) $n$-exact for $\lambda$, then it is also weakly (parametrically) $n$-exact for $\lambda$.

The equivalence of the existence of $n$-exact cardinals with $\ESR$ principles for $\Sigma_n$-definable classes of structures is given in the following theorem.

\begin{theorem}\label{theorem:CharacterizationSigma-ESR}
 The following statements are equivalent for all cardinals $\kappa$ and all natural numbers $n>0$: 
 \begin{enumerate}
     \item $\kappa$ is the least  cardinal such that $\Sigma_{n+1}$-$\ESR(\kappa)$ holds. 
     
     \item $\kappa$ is the least cardinal such that $\Sigma_{n+1}(V_\kappa)$-$\ESR(\kappa)$ holds. 
     
     \item $\kappa$ is the least cardinal that is  $n$-exact for some  $\lambda>\kappa$. 
     
     \item $\kappa$ is the least cardinal that is  parametrically  $n$-exact for some  $\lambda>\kappa$. 
 \end{enumerate}
\end{theorem}

The above results allow us to compare the large cardinal properties introduced in Definitions \ref{definition:WeaklyExactCardinalUp} and \ref{defnexact}.   
More specifically, if $\kappa$ is a cardinal satisfying the equivalent statements of Theorem \ref{theorem:CharacterizationSigma-ESR} for some $0<n<\omega$, then there exists a cardinal $\mu<\kappa$ satisfying the equivalent statements of Theorem \ref{theorem:CharacterizationPI-ESR} for the same natural number $n$ (Lemma \ref{lemma:CharESREmbSigma2}). 
 %
 %
 This implication should be compared with the corresponding statements for the principle $\SR$, showing that $\SR$ for $\Pi_n$-definable classes of structures is  equivalent to $\SR$ for $\Sigma_{n+1}$-definable classes ({\cite[Section 4]{Ba:CC}}).

Exact cardinals   are very strong, consistency-wise. 
In Section \ref{section:strength} we give lower and upper bounds for their consistency strength, and   we also prove they imply the existence of well-known large cardinals in the upper ranges of the large-cardinal  hierarchy. 
Recall that a cardinal $\kappa$ is \emph{almost huge} (see \cite{Kan:THI}) if it is the critical point of an elementary embedding $j:V\to M$, with $M$ transitive and closed under sequences of length less than $j(\kappa)$. Given such an embedding $j$, we then say that $\kappa$ is \emph{almost huge with target $j(\kappa)$}.  
If $\kappa$ is either parametrically $0$-exact for $\lambda$, or weakly parametrically $1$-exact for $\lambda$, then many cardinals smaller than $\kappa$ are almost huge with target $\kappa$ (Corollary \ref{corollary:AlmostHugeFromPi1ic}). As for upper bounds, while every huge cardinal (with target some $\lambda$) is weakly parametrically $1$-exact (for the same $\lambda$), the least huge cardinal $\kappa$ is not $1$-exact for any $\lambda >\kappa$ (Propositions \ref{proposition:ESRfromHuge} and  \ref{proposition:LeastHugeFailure}). A strong consistency upper bound is provided by an $I3$-embedding (see \cite[\S 24]{Kan:THI}), for if $j:V_\delta \to V_\delta$ is such an embedding, then in $V_\delta$ a proper class of cardinals are parametrically $n$-exact for unboundedly-many $\lambda$, for every $n$ (Proposition \ref{proposition:upperbound}). A much lower upper bound, namely an almost $2$-huge cardinal, is given in Proposition \ref{proposition:ESRfromHugely2Exact} for the consistency of weakly parametrically $n$-exact cardinals, all $n>0$.

Finally, in Section \ref{section:Beyond}, we show how the principle $\ESR(\kappa,\lambda)$ can be strengthened to encompass increasing sequences of cardinals of length at most $\omega$, instead of a single cardinal $\lambda$, in order  to obtain principles of structural reflection that are much stronger, implying the existence of many-times huge cardinals or even $I3$-embeddings. 
 The formulation of these stronger sequential $\ESR$  principles is motivated by the observation that 
 the principle $\Pi_1$-$\ESR(\kappa,\lambda)$ directly implies the instance $$(\lambda,\kappa) ~ \twoheadrightarrow ~ (\kappa,{<}\kappa)$$ of \emph{Chang's Conjecture}, i.e. every structure $A$ in a countable language with domain $\lambda$ has an elementary substructure $B$ of cardinality $\kappa$ with $\betrag{B\cap\kappa}<\kappa$. 
 The definition of our sequential $\ESR$ principles will then directly imply that higher versions of Chang's Conjecture  hold for the respective cardinals.

 We then also strengthen, accordingly, the notions of weakly exact and exact cardinals to obtain large cardinal properties that correspond to the new sequential $\ESR$ principles and show that much of the theory developed for $\ESR(\kappa,\lambda)$ can be generalized to this stronger context. In particular, we obtain exact equivalences for the least cardinals witnessing the sequential forms of $\ESR$ and the corresponding sequential forms of weakly exact and exact cardinals (Theorems \ref{theorem:SeqEquivalenceWeaklyExactReflection} and \ref{theorem:SeqEquivalenceExactReflection}). As for determining the position of these large cardinals in the large cardinal hierarchy, we show, on the one hand, that the existence of a  weakly $1$-exact or a $0$-exact cardinal for a  sequence of cardinals of length $n+1$ implies the existence of smaller $n$-huge cardinals.  On the other hand, every $n$-huge cardinal is weakly parametrically $n$-exact for some sequence of cardinals of length $n$  (Proposition \ref{proposition:SequentialWeaklyExactImplynHuge}). Also, if $\kappa$ is the critical point of an  $I1$-embedding (see \cite[\S 24]{Kan:THI}), then it is weakly parametrically $1$-exact for a sequence of cardinals of length $\omega$.

 Many  questions remain, and some of them are addressed in the last section of the article. Most interesting is the problem of determining the exact strength of the sequential forms of $\ESR$. We know that these principles, in the case of sequences of length $\omega$, are very strong, so much so that even when restricted to $\Sigma_2$-definable classes of structures we don't know them to be consistent. This makes the study of such principles both challenging and exciting, for they appear to constitute a new class of large-cardinal principles that go beyond $I1$-embeddings, yet they may not fall into \emph{Kunen's Inconsistency}.


\section{Isomorphism-Closed classes}

We start by studying instances of the principle $\ESR$ for classes of structures closed under isomorphic copies. 
Notice that if a class $\Ce$ of structures of the same type is $\Sigma_n$-definable (with or without parameters) for some $n>0$, then the closure of $\Ce$ under isomorphic copies is also $\Sigma_n$-definable (with the same parameters, if any).

\begin{proposition}\label{proposition:CharClosedIso}
 Given uncountable cardinals $\kappa<\lambda$, the following statements are equivalent for every class $\Ce$ of structures of the same type that is closed under isomorphic copies: 
 \begin{enumerate}
     \item $\ESR_{\Ce}{(\kappa,\lambda)}$. 
     
     \item For every structure $B$ in $\Ce$ whose cardinality is contained in the interval $[\cof{\lambda},\beth_\lambda]$, there exists an elementary embedding of a structure $A$ in $\Ce$ into $B$ such that the cardinality of $A$ is contained in the interval $[\cof{\kappa},\beth_\kappa]$. 
 \end{enumerate}
\end{proposition}

\begin{proof}
 Assume that (i) holds and fix a structure $B$ in $\Ce$ whose cardinality is contained in the interval $[\cof{\lambda},\beth_\lambda]$. 
 Then we can pick an injection $i$ from the domain of $B$ into $V_\lambda$ such that the set $$\Set{\gamma<\lambda}{\ran{i}\cap(V_{\gamma+1}\setminus V_\gamma)\ne \emptyset}$$ is unbounded in $\lambda$. 
  Let $B_0$ denote the isomorphic copy of $B$ induced by $i$. Then $B_0$ is a structure in $\Ce$ of rank $\lambda$ and our assumptions yield an elementary embedding of a structure $A$ in $\Ce$ of rank $\kappa$ into $B_0$. 
  But this allows us to conclude that the cardinality of $A$ is contained in the interval $[\cof{\kappa},\beth_\kappa]$, and there exists an elementary embedding of $A$ into $B$. 
  
  Now, assume that (ii) holds and fix a structure $B$ in $\Ce$ of rank $\lambda$. 
  Then the cardinality of $B$ is contained in the interval $[\cof{\lambda},\beth_\lambda]$ and our assumption yields an elementary embedding of a structure $A$ in $\Ce$ into $B$ whose cardinality is contained in the interval $[\cof{\kappa},\beth_\kappa]$. 
  Pick an injection $i$ from the domain of $A$ into $V_\kappa$ with the property that the set  $$\Set{\alpha<\kappa}{\ran{i}\cap(V_{\alpha+1}\setminus V_\alpha)\ne \emptyset}$$ is unbounded in $\kappa$, and let $A_0$ denote the isomorphic copy of $A$ induced by $i$. 
  Then $A_0$ is a structure in $\Ce$ of rank $\kappa$ and there exists an elementary embedding of $A_0$ into $B$. 
\end{proof}

\begin{corollary}
 Let $\kappa <\lambda$ be inaccessible cardinals and let  $\Ce$ be a class of structures of the same type that is closed under isomorphic copies. 
 Then $\ESR_{\Ce}{(\kappa,\lambda)}$ holds if and only if  for every structure $B\in\Ce$ of cardinality $\lambda$, there exists an elementary embedding of a structure $A\in\Ce$ of cardinality $\kappa$ into $B$. \qed
\end{corollary}

\begin{corollary}\label{corollary:UpDownForIsoClosed}
 Let $\Ce$ be a class of structures of the same type that is closed under isomorphic copies and let $\kappa<\mu<\lambda$ be infinite cardinals with the property that $\ESR_{\Ce}{(\kappa,\lambda)}$ holds.  
 \begin{enumerate}
     \item If $\cof{\mu}\leq\cof{\kappa}$, then $\ESR_{\Ce}{(\mu,\lambda)}$ holds. 
     
     \item If $\cof{\mu}\geq\cof{\lambda}$, then $\ESR_{\Ce}{(\kappa,\mu)}$ holds. 
 \end{enumerate} 
\end{corollary}
 
\begin{proof}
 Since $\cof{\mu}\leq\cof{\kappa}$ implies $[\cof{\kappa},\beth_\kappa]\subseteq[\cof{\mu},\beth_\mu]$ and $\cof{\mu}\geq\cof{\lambda}$ implies $[\cof{\mu},\beth_\mu]\subseteq[\cof{\lambda},\beth_\lambda]$, both statements follow directly from Proposition \ref{proposition:CharClosedIso}.  
\end{proof}


\section{Low complexities}

In this section, we study exact structural reflection for $\Sigma_1$-definable classes of structures. 
 In the case of classes closed under isomorphic copies, these principles are provable in \rm{ZFC}.

\begin{proposition}\label{proposition:Sigma1ClosedTrivial}
 If $\kappa$ is an uncountable cardinal, then the principle $\Sigma_1(V_\kappa)^{ic}$-$\ESR(\kappa,\lambda)$ holds for every cardinal $\lambda>\kappa$ with $\cof{\kappa}\leq\cof{\lambda}$.  
\end{proposition}

\begin{proof}
 Fix a $\Sigma_1$-formula $\varphi(v_0,v_1)$ and $z\in V_\kappa$ such that  $\Ce=\Set{A}{\varphi(A,z)}$ is a class of structures of the same type and pick  a structure $B$ in $\Ce$ whose cardinality is contained in the interval  $[\cof{\lambda},\beth_\lambda]$.
 Let $B_0$ be an isomorphic copy of $B$  in $H_{\beth_\lambda^+}$  and pick an elementary substructure $X$ of $H_{\beth_\lambda^+}$ of cardinality $\beth_\kappa$ with $V_\kappa\cup\{\kappa,B_0\} \subseteq X$. 
 Let $\pi:X\to M$ denote the induced transitive collapse. 
 Since $\pi(z)=z$, $\Sigma_1$-absoluteness now implies that $\varphi(\pi(B_0),z)$ holds and hence $\pi(B_0)$ is an element of $\Ce$. 
 Moreover,  our construction ensures that $\pi(B_0)$ has cardinality at most $\beth_\kappa$ and, since $B_0$ has cardinality at least $\cof{\lambda}\geq \cof{\kappa}$ and $\pi(\cof{\kappa})=\cof{\kappa}$, we know that $\pi(B_0)$ has cardinality at least $\cof{\kappa}$. 
 Finally, using the inverse collapse $\pi^{{-}1}$, it is easy to see that there exists an elementary embedding of $\pi(B_0)$ into $B$. 
 By Proposition \ref{proposition:CharClosedIso}, the above computations yield the desired conclusion. 
\end{proof}

In contrast with the previous Proposition, the principle  $\ESR_{\Ce}(\kappa,\lambda)$  for some $\kappa <\lambda$ and all $\Sigma_0$-definable (without parameters) classes $\Ce$ of structures of the same type (so, no closure under isomorphic copies required), has considerable large-cardinal strength  and fails in G\"odel's constructible universe $\LL$.

\begin{lemma}\label{lemma:ESRSigma0}
 If $\Sigma_0$-$\ESR{(\kappa)}$ holds for some uncountable cardinal  $\kappa$, then $a^\#$ exists for every real $a$. 
\end{lemma}

\begin{proof}
  Let $\calL$ denote the first-order language that extends the language $\calL_\in$ of set theory by  a binary predicate symbol $\dot{E}$,  a  constant symbol $\dot{c}$ and a unary function symbol $\dot{f}$. 
  Define $\Ce$ to be the class of all $\calL$-structures of the form $\langle\nu,\in,E,\alpha,f\rangle$ with the property that $\nu$ is an ordinal and $f:\langle\nu,\in\rangle\to\langle \ran{f},E\rangle$ is an order-isomorphism. 
  Then it is easy to see that $\Ce$ is definable by a $\Sigma_0$-formula without parameters. 
  
  Now, fix a real $a$ and a cardinal $\lambda>\kappa$ such that $\ESR_\Ce(\kappa,\lambda)$ holds. Pick a bijection $b:\LL_\lambda[a]\to\lambda$ and set $$E ~ = ~ \Set{\langle b(x), b(y)\rangle}{x\in y\in\LL_\lambda[a]}.$$ 
  Then $$B ~ = ~ \langle\lambda,\in,E,b(\kappa),{b\restriction\lambda}\rangle$$ is an $\calL$-structure of  rank $\lambda$ in $\Ce$. 
  By our assumption, there exists a binary relation $R$ on $\kappa$, a function $f:\kappa\to\kappa$ and $\alpha<\kappa$ such that $$A ~ = ~ \langle\kappa,\in,R,\alpha,f\rangle$$ is a structure in $\Ce$ with the property that there exists an elementary embedding $i$ of $A$ into $B$. 
  Since our construction ensures that $\langle\lambda,E\rangle$ is well-founded and $\langle \kappa,R\rangle$ embeds into $\langle\lambda,E\rangle$, it follows that $\langle \kappa,R\rangle$ is  well-founded too. 
  Moreover, elementarity implies that $\langle\kappa,R\rangle$ is extensional. 
  Let $\pi:\langle\kappa,R\rangle\to\langle M,\in\rangle$ denote the corresponding transitive collapse and set $$j ~ = ~ b^{{-}1}\circ i\circ\pi^{{-}1}:M\to\LL_\lambda[a].$$ Then $j$ is an elementary embedding of transitive structures. 
  
  Now, note that elementarity implies that $\ran{f}=\pi^{{-}1}[M\cap\On]$ and $$\pi\restriction \ran{f}:\langle \ran{f},R\rangle\to\langle M\cap\On,\in\rangle$$ is an order-isomorphism.  
  But this shows that $$\pi\circ f:\langle \kappa,\in\rangle\to\langle M\cap\On,\in\rangle$$ is also an order-isomorphism and hence we can conclude that $M\cap\On=\kappa$. 
  In particular, elementarity implies that $M=\LL_\kappa[a]$. 
  
  Finally, since $j(\pi(\alpha))=\kappa>\pi(\alpha)$, we know that $j:\LL_\kappa[a]\to\LL_\lambda[a]$ is a  non-trivial elementary embedding.  
  But then $\betrag{\crit{j}}<\kappa$ and the proof of {\cite[Theorem  18.27]{Jech}} shows that $a^\#$ exists.  
\end{proof}

The next result provides an upper bound for the consistency strength of the assumption of Lemma \ref{lemma:ESRSigma0}.  %
In particular, it shows that this  assumption  does not imply the existence of an inner model with a measurable cardinal. 
Its proof is based on arguments contained in  the proof of {\cite[Theorem 2.3]{MR1856729}}.

\begin{lemma}
 If $\delta$ is a Ramsey cardinal, then the set of inaccessible cardinals $\kappa<\delta$ with the property that $\Sigma_1(V_\kappa)$-$\ESR{(\kappa)}$  holds in $V_\delta$ is unbounded in $\delta$. 
\end{lemma}

\begin{proof}
 Fix $\xi<\delta$ and pick $A\subseteq\delta$ such that $V_\delta=\LL_\delta[A]$. 
 By our assumption, there exists a good set $I$ of indiscernibles for the structure $\langle\LL_\delta[A],\in,A\rangle$ (see {\cite[Section 1]{MR645907}}) that is unbounded in $\delta$ and satisfies $\min(I)>\xi$. 
 Set $$X ~ = ~ \mathrm{Hull}_{\langle\LL_\delta[A],\in,A\rangle}(\min(I)\cup(I\setminus\{\min(I)\}))$$ and let $\pi:X\to M$ denote the corresponding transitive collapse. 
 Since $I$ is unbounded in $\delta$, we know that $M\cap\On=\delta$. 
 Moreover,  indiscernibility ensures that $\min(I)\notin X$ and hence $\pi^{{-}1}:M\to V_\delta$ is a non-trivial elementary embedding with critical point $\min(I)$. Set $\kappa=\pi^{{-}1}(\min(I))\in M$ and $\lambda=\pi^{{-1}}(\kappa)\in M$. 
 Then $\kappa$ and $\lambda$ are both inaccessible cardinals greater than $\xi$.

 \begin{claim*}
  In $V_\delta$, the principle $\ESR_\Ce(\kappa,\lambda)$ holds for every class $\Ce$ of structures of the same type that is definable by a $\Sigma_1$-formula with parameters in  $V_\kappa$. 
 \end{claim*}
 
 \begin{proof}[Proof of the Claim]
  Assume, towards a contradiction, that there exists a $\Sigma_1$-formula $\varphi(v_0,v_1)$ with the property that for some $z\in V_\kappa$, the class $\Ce=\Set{A\in V_\delta}{\varphi(A,z)}$ consists of structures of the same type and there exists $B\in\Ce$ of rank $\lambda$ such that for all $A\in\Ce$ of rank $\kappa$, there is no elementary embedding of $A$ into $B$. %
  Using elementarity, we now know that, in $M$, there exist $z_0\in V_{\min(I)}$ and $B_0\in V_{\kappa+1}\setminus V_\kappa$ with the property that the class $\Ce_0=\Set{A\in M}{\varphi(A,z_0)}$ consists of structures of the same type, $B_0\in\Ce_0$ and for all $A\in\Ce_0$ of rank $\min(I)$, there is no elementary embedding of $A$ into $B_0$.
  Since we have $$\pi^{{-}1}\restriction V_{\min(I)}^M ~ = ~ \id_{V_{\min(I)}^M},$$ the elementarity of $\pi^{{-}1}$ implies that $\Ce_*=\Set{A\in   V_\delta}{\varphi(A,z_0)}$ consists of structures of the same type, $\pi^{{-}1}(B_0)\in\Ce_*$ and for all $A\in\Ce_*$ of rank $\kappa$, there is no elementary embedding of $A$ into $\pi^{{-}1}(B_0)$.
  But this yields a contradiction, because  the upwards absoluteness of $\Sigma_1$-formulas  implies that $B_0$ is a structure in $\Ce_*$ of rank $\kappa$ and $\pi^{{-}1}$ induces an elementary embedding of $B_0$ into $\pi^{{-}1}(B_0)$. 
 \end{proof}
 
 The above claim completes the proof of the lemma. 
\end{proof}


\section{The $\Pi_n$-case for isomorphism-closed classes}

We now show that the structural reflection principles introduced in Definition \ref{definition:ESR} become very strong when they hold for more complex classes of structures.  
 In particular, the validity of these principles for  $\Pi_1$-definable classes $\Ce$ of structures closed under isomorphic copies at singular cardinals already implies non-trivial fragments of supercompactness, and the corresponding principles for a regular cardinal will turn out to imply the existence of many almost huge cardinals.

 Recall that a cardinal $\kappa$ is \emph{$\lambda$-supercompact} if  there is a transitive class $M$ closed under $\lambda$-sequences and an elementary embedding $j:V\to M$ with  $\crit{j}=\kappa$ and $j(\kappa)>\lambda$.

\begin{lemma}
\label{lemma:singularsupercompact}
 Let  $\kappa<\lambda$ be infinite cardinals such that $\kappa$ is singular and $\cof{\kappa}\leq\cof{\lambda}$. 
 If the interval $(\cof{\kappa},\kappa)$ contains  a $\beth_\lambda$-supercompact cardinal $\mu$, then $\Pi_1(V_\mu)^{ic}$-$\ESR{(\kappa,\lambda)}$ holds. 
\end{lemma}

\begin{proof}
  Fix a $\Pi_1$-formula $\varphi(v_0,v_1)$ and an element $z$ of $V_\mu$ such that $\Ce=\Set{A}{\varphi(A,z)}$ is a class of structures of the same type that is closed under isomorphic copies. 
  Pick a structure $B$ in $\Ce$ whose cardinality is contained in the interval $[\cof{\lambda},\beth_\lambda]$ and whose domain is a subset of $\beth_\lambda$. 
  Now, fix  an elementary embedding $j:V\to M$ with $\crit{j}=\mu$, $j(\mu)>\beth_\lambda$ and ${}^{\beth_\lambda}M\subseteq M$. 
   Then the closure properties of $M$ ensure that $B$ is an element of $M$ and the elementary embedding of $B$ into $j(B)$ induced by $j$ is also contained in $M$. 
   Moreover, $\Pi_1$-downwards absoluteness for transitive classes implies that $\varphi(B,z)$ holds in $M$. Thus, $M$ satisfies that there exists a structure $A$  whose cardinality is  contained in the interval $[\cof{\kappa}, j(\mu))$ such that $\varphi(A,z)$ holds, and there exists an elementary embedding $e:A\to j(B)$, as this is witnessed by $B$ and $j\restriction B:B\to j(B)$. 
   Since $j(z)=z$ and $j(\cof{\kappa})=\cof{\kappa}$, the elementarity of $j$ yields an elementary embedding of a structure $A$ of the given type into $B$ such that $\varphi(A,z)$ holds and the cardinality of $A$  is contained in the interval $[\cof{\kappa},\mu)$. 
  This shows that there is an elementary embedding of a structure in $\Ce$ whose cardinality is contained in  the interval $[\cof{\kappa},\beth_\kappa]$ into $B$. 
   By Proposition \ref{proposition:CharClosedIso}, this proves the lemma. 
\end{proof}

\begin{corollary}\label{corollary:ESR-Pi-1-Iso}
 If $\kappa<\lambda$ are cardinals such that $\cof{\kappa}\leq\cof{\lambda}$ and $\kappa$ is a singular limit of $\beth_\lambda$-supercompact  cardinals, then  $\Pi_1(V_\kappa)^{ic}$-$\ESR{(\kappa,\lambda)}$ holds. 
 Hence,  if $\kappa$ is a singular limit of supercompact cardinals, then  $\Pi_1(V_\kappa)^{ic}$-$\ESR{(\kappa,\lambda)}$ holds for a proper class of cardinals $\lambda$. \qed 
\end{corollary}

We continue by showing that, for all $n>1$,  analogous statements hold for $\Pi_{n}$-definable classes and $[\beth_\kappa, \beth_\lambda +1)$-$C^{(n-1)}$-extendible  cardinals.

\begin{lemma}
\label{lemma:singularextendible}
 Let $n>0$ be a natural number and let $\kappa<\lambda$ be infinite cardinals such that $\kappa$ is singular and $\cof{\kappa}\leq\cof{\lambda}$ holds. 
 If the interval $(\cof{\kappa},\kappa)$ contains a cardinal  $\delta$ that is $[\mu, \beth_\lambda +1)$-$C^{(n)}$-extendible for some ordinal $\delta\leq\mu\leq\beth_\kappa$, then $\Pi_{n+1}(V_\delta)^{ic}$-$\ESR{(\kappa,\lambda)}$ holds. 
\end{lemma}

\begin{proof}
  Fix a $\Pi_{n+1}$-formula $\varphi(v_0,v_1)$ and an element $z$ of $V_\delta$ such that $\Ce=\Set{A}{\varphi(A,z)}$ is a class of structures of the same type that is closed under isomorphic copies. 
  Pick a structure $B$ in $\Ce$ whose cardinality is contained in the interval $[\cof{\lambda},\beth_\lambda]$ and whose domain is a subset of $\beth_\lambda$. 
  Let $\nu \in C^{(n+1)}\cap [\mu, \beth_\lambda +1)$ and let  $j:V_{\beth_\lambda +1}\to V_\eta$ be an elementary embedding with $\crit{j}=\delta$, $j(\mu)\geq \beth_\lambda +1$,    and $j(\nu)\in C^{(n)}$. 
   Then the fact that $\beth_\lambda<j(\mu)\leq j(\nu)$ implies that $B$ is an element of $V_{j(\nu)}$ and the elementary embedding of $B$ into $j(B)$ induced by $j$ is  contained in $V_{\eta}$.
   Moreover, since $j(\nu)\in C^{(n)}$, by $\Pi_{n+1}$-downwards absoluteness for $V_{j(\nu)}$, we have  that $\varphi(B,z)$ holds in $V_{j(\nu)}$. 
   Thus, $V_{\eta}$ satisfies that there exists a structure $A$  whose  cardinality is contained in the interval $[\cof{\kappa}, j(\mu))$ such that $\varphi(A,z)$ holds in $V_{j(\nu)}$, and there exists an elementary embedding $e:A\to j(B)$, as this is witnessed by $B$ and ${j\restriction B}:B\to j(B)$. 
   Since $j(z)=z$ and $j(\cof{\kappa})=\cof{\kappa}$, the elementarity of $j$ yields an elementary embedding of a structure $A$ of the given type into $B$ such that $\varphi(A,z)$ holds in $V_\nu$ and the cardinality of $A$  is contained in the interval $[\cof{\kappa},\mu)$.  But since $\nu\in C^{(n+1)}$, we know that $\varphi(A,z)$ also holds in $V$.
  This shows that there is an elementary embedding of a structure in $\Ce$ whose cardinality is contained in  the interval $[\cof{\kappa},\beth_\kappa]$ into $B$. 
   By Proposition \ref{proposition:CharClosedIso}, this proves the lemma. 
\end{proof}

The following observation shows that $C^{(n)}$- extendible cardinals provide natural examples of $[\mu,\infty)$-$C^{(n)}$-extendible cardinals.

\begin{proposition}\label{proposition:CnExtendibleIntervalExtendible}
$ $
\begin{enumerate}
    \item  If $\kappa$ is a  $\lambda$-$C^{(n)}$-extendible cardinal and  $C^{(n+1)}\cap[\kappa ,\lambda)\neq \emptyset$, 
  then $\kappa$ is $[\kappa, \lambda)$-$C^{(n)}$-extendible. 
  \item Every $C^{(n)}$-extendible cardinal $\kappa$ is $[\mu  ,\infty)$-$C^{(n)}$-extendible, for every $\mu \geq \kappa$.
  \end{enumerate}
\end{proposition}

\begin{proof} 
(i):  Assume $\kappa$ is  $\lambda$-$C^{(n)}$-extendible and $\nu \in C^{(n+1)}\cap[\kappa ,\lambda )$. Since $\kappa\in C^{(n)}$, and since every true $\Sigma_{n+1}$ statement, with parameters in $V_\kappa$ is true in $V_\nu$, the assumption that $\kappa$ is $\lambda$-$C^{(n)}$-extendible easily yields that $\kappa \in C^{(n+1)}$. Then $\kappa$ itself witnesses the $[\kappa ,\lambda)$-$C^{(n)}$-extendibility of $\kappa$.

(ii):
 As  shown in \cite{Tsan}, a cardinal $\kappa$ is $C^{(n)}$-extendible if and only if it is ${C^{(n)}}^+$-extendible, i.e., for a proper class of  $\lambda\in C^{(n)}$ there exists an elementary embedding $j:V_\lambda \to V_\eta$ for some $\eta \in C^{(n)}$, with $\crit{j}=\kappa$,  $j(\kappa)>\lambda$ and $j(\kappa)\in C^{(n)}$. Thus, if $\kappa$ is $C^{(n)}$-extendible, then it is $[\mu ,\lambda)$-$C^{(n)}$-extendible, for every $\mu\geq \kappa$ and every $\lambda\in C^{(n)}$ such that $C^{(n+1)}\cap [\mu,\lambda)\ne \emptyset$. In particular, every $C^{(n)}$-extendible cardinal $\kappa$ is $[\mu ,\infty)$-$C^{(n)}$-extendible for every $\mu \geq \kappa$.
\end{proof}

\begin{corollary}\label{corollary:LimitsOfCnExtendibles}
  Let $n>0$ be a natural number, let  $\kappa$ be a singular cardinal and let $\lambda>\kappa$ be a cardinal with $\cof{\kappa}\leq\cof{\lambda}$. 
  \begin{enumerate}
      \item If $\kappa$ is a limit of cardinals $\delta$ that are $[\mu,{\beth_\lambda+1})$-$C^{(n)}$-extendible for some $\delta\leq\mu\leq\kappa$, then  $\Pi_{n+1}(V_\kappa)^{ic}$-$\ESR{(\kappa,\lambda)}$ holds.
      
      \item If $\kappa$ is a limit point of $C^{(n+1)}$ and a limit of $(\beth_\lambda +1)$-$C^{(n)}$-extendible cardinals, then $\Pi_{n+1}(V_\kappa)^{ic}$-$\ESR{(\kappa,\lambda)}$ holds. 
      
      \item If $\kappa$ is  a  limit of $C^{(n)}$-extendible cardinals, then    $\Pi_{n+1}(V_\kappa)^{ic}$-$\ESR{(\kappa,\lambda)}$ holds
  \end{enumerate}
\end{corollary}

\begin{proof}
 The first statement follows directly from Lemma \ref{lemma:singularextendible}. 
 For the second statement, notice that our assumption implies that $\kappa$ is a limit of $(\beth_\lambda +1)$-$C^{(n)}$-extendible cardinals $\delta$ with the property that $C^{(n+1)}\cap(\delta,\kappa)\neq\emptyset$. By Proposition \ref{proposition:CnExtendibleIntervalExtendible}, this implies that $\kappa$ is a limit of cardinals $\delta$ that are $[\delta,{\beth_\lambda+1})$-$C^{(n)}$-extendible and therefore we can use the first part to derive the desired conclusion. 
 Finally, since {\cite[Proposition 3.4]{Ba:CC}} shows that all $C^{(n)}$-extendible cardinals are elements of  $C^{(n+2)}$, we can apply the second part of the corollary to prove the third statement.  
 \end{proof}

 We are now ready to show that Vop\v{e}nka's Principle can be characterized through principles of exact structural reflection.

\begin{proof}[Proof of Theorem \ref{theorem:ESR-VP}]
 First, assume that (i) holds and fix a natural number $n>0$. 
 Since {\cite[Corollary 6.9]{BCMR}} shows that our assumption implies the existence of a proper class of $C^{(n)}$-extendible cardinals, there exists a proper class of cardinals $\kappa$ of countable cofinality that are limits of $C^{(n)}$-extendible cardinals. The third part of Corollary \ref{corollary:LimitsOfCnExtendibles} now implies that for every such cardinal $\kappa$ and every $\lambda>\kappa$, the principle  $\Pi_{n+1}(V_\kappa)^{ic}$-$\ESR(\kappa,\lambda)$ holds. This shows that  (iii) holds in this case.  
 
 Now, assume that (ii)
 holds and let $\Ce_0$ be a proper class of structures of the same type. Let $\Ce$ denote the class of all structures of the given type that are isomorphic to a structure in $\Ce_0$ and define $C=\Set{\betrag{A}}{A\in\Ce}$. 
 By our assumptions, there exists a  cardinal $\kappa$ with  the property that $\ESR_\Ce(\kappa,\lambda)$ holds for all  $\lambda>\kappa$. 
 If $C$ is a proper class, then there exists a structure $B\in\Ce_0$ of cardinality greater than $\beth_\kappa$ and, since the principle  $\ESR_\Ce(\kappa,\vert B\vert)$ holds, we can use Proposition \ref{proposition:CharClosedIso} to find a structure $A\in\Ce_0$ of cardinality at most $\beth_\kappa$ and an elementary embedding of $A$ into $B$. 
 In particular, we know that Vop\v{e}nka's Principle for $\Ce$ holds in this case. 
 In the other case, namely if $C$ is a set, then we can find distinct  $A,B\in\Ce_0$ that are isomorphic and hence Vop\v{e}nka's Principle for $\Ce$ also holds in this case. This allows us to conclude that (i) holds. 
 
 This concludes the proof of the theorem, because (iii) obviously implies (ii).  
\end{proof}

 Similar results hold also for $\Sigma_n(V_\kappa)$-definable classes closed under isomorphic copies, assuming $\kappa$ is  a singular limit of supercompact cardinals, in the case $n=2$,  or a singular limit of $[\kappa,\infty)$-$C^{(n-2)}$-extendible cardinals, in the case $n>2$. 

\begin{corollary}\label{corollary:ESRfromSRatSingulars}
 Let $\kappa$ be a singular cardinal and let $\lambda>\kappa$ be a cardinal with $\cof{\kappa}\leq\cof{\lambda}$. 
 \begin{enumerate}
     \item If $\kappa$ is a limit of supercompact cardinals, then $\Sigma_2(V_\kappa)^{ic}$-$\ESR(\kappa,\lambda)$ holds. 
     
     \item If $n>0$   and $\kappa$ is a limit of $[\kappa, \infty)$-$C^{(n)}$-extendible cardinals, then $\Sigma_{n+2}(V_\kappa)^{ic}$-$\ESR(\kappa,\lambda)$ holds. 
 \end{enumerate}
\end{corollary}

\begin{proof}
(i)  Assume that $\kappa$ is a limit of supercompact cardinals and fix $z\in V_\kappa$. Let $\varphi(v_0,v_1)$ be a $\Sigma_2$-formula such that $\Ce=\Set{A}{\varphi(A,z)}$ is a class of structures of the same type that is closed under isomorphic copies. Pick a structure $B$ in $\Ce$ whose cardinality is contained in the interval $[\cof{\lambda},\beth_\lambda]$.  We can use our assumption to  find cardinals $\mu<\kappa$ and $\theta\geq\beth_\lambda$ such that $z\in V_\mu$ and there exists a transitive class $M$ containing $B$ and closed under $\theta$-sequences, and an elementary embedding $j:V\to M$ with $\crit{j}=\mu$, $j(\mu)>\theta$ and the property that $\varphi(B,z)$ holds in $M$. 
 In this situation, we can repeat the proof of Lemma \ref{lemma:singularsupercompact} to find a structure $A$ in $\Ce$ whose cardinality is contained in the interval $[\cof{\kappa},\beth_\kappa]$ and an elementary embedding of $A$ into $B$. 

(ii) Now, assume that $\kappa$ is a limit of $[\kappa,\infty)$-$C^{(n)}$-extendible cardinals. Fix a $\Sigma_{n+2}$-formula $\varphi(v_0,v_1)$ and $z\in V_\kappa$ that define  a class $\Ce$ of structures of the same type closed under isomorphic copies. 
 If we now pick a structure $B\in\Ce$ whose cardinality is contained in the interval $[\cof{\lambda},\beth_\lambda]$, then there is a cardinal $\beth_\lambda<\theta\in C^{(n+2)}$, a cardinal $\nu\in C^{(n+1)}\cap[\kappa,\theta)$ and an elementary embedding $j:V_\theta\to V_\eta$ for some $\eta$ such that $z\in V_{\crit{j}}$,  $j(\kappa)\geq\theta$ and $j(\nu)\in C^{(n)}$. 
  Using the fact that $\theta\in C^{(n+2)}$, we can now continue as in the proof of Lemma \ref{lemma:singularextendible} to obtain the desired elementary embedding. 
\end{proof}

We shall next prove several results that will allow us to derive high lower bounds for the consistency strength of the principle $\Pi_1^{ic}$-$\ESR$. 
 In particular, these results will show that passing from $\Sigma_1$-definable to $\Pi_1$-definable classes of structures drastically increases the strength of the principle $\ESR$.

\medskip

Given a set $z$ and a natural number $n>0$, we let $\mathcal{W}_n(z)$ denote the class of all structures (in the language $\calL_\in$ of set theory extended by five constant symbols and two unary function symbols) of the form  $\langle D,E,a,b,c,d,e, f,g\rangle$ with the property  that the relation $E$ is well-founded and extensional, and, if $\pi:\langle D,E\rangle\to\langle M,\in\rangle$ denotes the corresponding transitive collapse, then the following statements hold: 
  \begin{enumerate}
      \item $\pi(a),\pi(b),\pi(c), \pi(d) \in M\cap\On$,  and $V_{\pi(c)}\subseteq M$. 
      
      \item If $\pi(b)>0$, then $\pi(b)\in C^{(n)}$. 
      
      \item $\pi(e)=z\in V_{\pi(a)}$. 
      
      \item $M$ is $\Pi_n(V_{\pi(c)+1})$-upwards correct. 
      
      \item The map $\pi\circ f\circ\pi^{{-}1}$ induces a bijection between $V_{\pi(c)}$ and $\pi(d)$.  
      
      \item The map $g\circ\pi^{{-}1}$ restricts to a bijection between $\pi(d)$ and $D$. 
  \end{enumerate}
  Then it is easy to see that the class $\mathcal{W}_n(z)$ is closed under isomorphic copies and  is definable by a $\Pi_n$-formula with parameter $z$.

\begin{lemma}\label{lemma:EmbFromESRisoC}
 Let $\mu<\lambda$ be cardinals such that  $\ESR_{\mathcal{W}_n(z)}(\mu,\lambda)$ holds for some element  $z$ of $V_\mu$ and some natural number $n>0$. 
 Given a cardinal $\lambda'\in C^{(n)}$ greater than $\beth_\lambda$, there exist 
 \begin{itemize}
     \item  a cardinal $\kappa\leq\mu$ with $z\in V_\kappa$ and   $\cof{\mu}\leq\beth_\kappa$, 
     
     \item a transitive, $\Pi_n(V_{\kappa+1})$-upwards correct set $M$ with $V_\kappa\cup\{\beth_\kappa\}\subseteq M$, and 
     
     \item an elementary embedding $j:M\to H_{\lambda'}$ with $\crit{j}<\kappa$, $j(\crit{j})\leq\mu$, $j(z)=z$, $\mu\in\ran{j}$ and  $j(\kappa)=\lambda$. 
 \end{itemize}
 In addition, for every $\nu\in C^{(n)}\cap[\mu,\lambda)$, we can find objects satisfying the above statements such that $j(\zeta)=\nu$ holds for some $\zeta\in C^{(n)}\cap \kappa$. 
\end{lemma}

\begin{proof}
 Pick  an elementary substructure $X$ of $H_{\lambda'}$ of cardinality $\beth_\lambda$ with $V_\lambda\cup(\beth_\lambda+1)\subseteq X$, a map $h_0:X\to X$ that extends a  bijection between $V_\lambda$ and $\beth_\lambda$, and a map $h_1:X\to X$ that extends a bijection between $\beth_\lambda$ and $X$. 
  Fix an ordinal $\nu$ such that either $\nu=0$ or $\nu\in C^{(n)}\cap[\mu,\lambda)$. 
  Since  $V_{\lambda+1}\cap X$ is contained in the transitive part of $X$, it follows that the transitive collapse of $X$ is $\Pi_n(V_{\lambda+1})$-upwards correct and hence $$\langle X,\in,\mu,\nu, \lambda,\beth_\lambda,z,h_0,h_1\rangle$$ is a structure in $\mathcal{W}_n(z)$ of cardinality $\beth_\lambda$. 
  By Proposition \ref{proposition:CharClosedIso}, our assumptions allow us to find a structure $\langle D,E,a,b,c,d,e, f,g\rangle$ in $\mathcal{W}_n(z)$ whose cardinality is contained in the interval $[\cof{\mu},\beth_\mu]$ and an elementary embedding $$i:\langle D,E,a,b,c,d,e, f,g\rangle\to\langle X,\in,\mu,\nu, \lambda,\beth_\lambda,z,h_0,h_1\rangle.$$ 
  Let $\pi:\langle D,E\rangle\to\langle M,\in\rangle$ denote the corresponding transitive collapse. Set $\theta=\pi(a)$, $\zeta=\pi(b)$, and $\kappa=\pi(c)$. 
  Then we have  $\pi(e)=z\in V_\theta\subseteq V_\kappa\subseteq M$, 
  $$\beth_\kappa ~ = ~ \vert 
  V_\kappa\vert ~ = ~ \vert M\vert ~ \in  ~ [\cof{\mu},\beth_\mu]$$ and therefore $\rank{z}<\theta<\kappa\leq\mu$.
 Moreover, since the ordinal $\pi(d)$ has cardinality $\beth_\kappa$,  our setup ensures that $\beth_\kappa\in M$.
 The definition of the class $\mathcal{W}_n(z)$ also ensures that $M$ is $\Pi_n(V_{\kappa+1})$-upwards correct. 
  Define $$j ~ = ~ i\circ\pi^{{-}1}:M\to H_{\lambda'}.$$ 
 Then $j$ is an elementary embedding satisfying $j(z)=z$,  $j(\theta)=\mu\geq\kappa>\theta$, and $j(\kappa)=\lambda$.   In particular, we know that $\crit{j}\leq\theta <\kappa$ and $j(\crit{j})\leq j(\theta)=\mu$. 
 Finally, if we have $\nu\in C^{(n)}\cap[\mu,\lambda)$, then elementarity implies that $\zeta>0$ and this allows us to conclude that $\zeta$ is an element of $C^{(n)}\cap\kappa$ with $j(\zeta)=\nu$.  
\end{proof}

The following direct consequence of the \emph{Kunen Inconsistency} (see {\cite[Corollary 23.14]{Kan:THI}}) will be used in our subsequent arguments:

\begin{proposition}\label{proposition:HighCritical}
 Given cardinals $\kappa<\lambda$ and an  ordinal $\alpha<\kappa$, if $j:V_\kappa\to V_\lambda$ is a non-trivial elementary embedding with $j(\alpha)=\alpha$, then $\crit{j}>\alpha$.  
\end{proposition}

\begin{proof}
 Assume, towards a contradiction, that $\crit{j}<\alpha$ holds. Then $j^n(\crit{j})<\alpha<\kappa$ holds for all $n<\omega$ and hence we can conclude that  $\rho=\sup_{n<\omega}j^n(\crit{j})\leq\alpha<\kappa$. Since $j(\rho)=\rho$ and $V_{\rho+2} \subseteq  V_{\alpha+2}  \subseteq V_\kappa$,  we know that $j\restriction V_{\rho+2}:V_{\rho+2}\to V_{\rho+2}$ is a non-trivial elementary embedding. This contradicts the \emph{Kunen Inconsistency}. 
\end{proof}

 The next results show that the assumptions of Corollary \ref{corollary:ESR-Pi-1-Iso} are close to optimal.

 \begin{lemma}\label{lemma:P1ESRisoSupercompacts}
  Let $\mu<\lambda$ be cardinals and let $\alpha<\mu$ be an ordinal. 
 \begin{enumerate}
     \item If  $\Pi_1(\{\alpha\})^{ic}$-$\ESR(\mu,\lambda)$ holds, then the interval $(\alpha,\mu]$ contains a  ${<}\lambda$-supercompact cardinal. 
  
  \item Given a natural number  $n>0$, if  $C^{(n+1)}\cap [\mu,\lambda)\neq\emptyset$ and $\Pi_{n+1}(\{\alpha\})^{ic}$-$\ESR(\mu,\lambda)$ holds, then the interval $(\alpha,\mu]$ contains a $[\mu,\lambda)$-$C^{(n)}$-extendible cardinal. 
  \end{enumerate}
\end{lemma}

\begin{proof}
(i)  Pick a cardinal $\lambda'>\beth_\lambda$ with the property that $H_{\lambda'}$ is sufficiently elementary in $V$. 
 Since $\ESR_{\mathcal{W}_1(\alpha)}(\mu,\lambda)$ holds, an application of Lemma \ref{lemma:EmbFromESRisoC} shows that there is a cardinal $\alpha<\kappa\leq\mu$, a transitive set $M$ with $V_\kappa\cup\{\kappa\}\subseteq M$ and an elementary embedding $j:M\to H_{\lambda'}$ with $\crit{j}<\kappa$, $j(\crit{j})\leq\mu$,  $j(\alpha)=\alpha$ and $j(\kappa)=\lambda$. 
 In this situation, Proposition \ref{proposition:HighCritical} shows that $\crit{j}>\alpha$. 
 By {\cite[Lemma 2]{Mag}}, the fact that $V_\kappa\subseteq M$ implies that $\crit{j}$ is ${<}\kappa$-supercompact. 
  Since all  ultrafilters witnessing this property are contained in $V_\kappa\subseteq M$, it follows that $\crit{j}$ is ${<}\kappa$-supercompact in $M$ and hence $j(\crit{j})\in(\alpha,\mu]$ is ${<}\lambda$-supercompact in both  $H_{\lambda'}$ and $V$. 
  
  (ii) Pick some $\nu \in C^{(n+1)}\cap [\mu, \lambda)$ and $\lambda' \in C^{(n+1)}$ greater than $\beth_\lambda$. 
  Since $\ESR_{\mathcal{W}_{n+1}(\alpha)}(\mu,\lambda)$ holds, Lemma \ref{lemma:EmbFromESRisoC} allows us to find a cardinal $\kappa$ with  $\alpha<\kappa\leq\mu$, a cardinal $\zeta\in C^{(n+1)}\cap\kappa$, a cardinal $\theta\leq\zeta$, 
  a $\Pi_{n+1}(V_{\kappa +1})$-upwards correct transitive set $M$ with $V_\kappa\cup\{\kappa\}\subseteq M$, and an elementary embedding $j:M\to H_{\lambda'}$ with $\crit{j}\leq\theta$,  $j(\alpha)=\alpha$,  $j(\theta)=\mu$, $j(\zeta)=\nu$  and $j(\kappa)=\lambda$. 
  By Proposition \ref{proposition:HighCritical}, we have $\crit{j}>\alpha$.
  Now, notice that $\lambda$ and $j\restriction V_\kappa$ witness that there exists an ordinal $\eta$ and an elementary embedding $i:V_\kappa\to V_\eta$ with $\crit{i}\in(\alpha,\theta]$, $i(\theta)\geq\kappa$ and $i(\zeta)\in C^{(n)}$, and this statement can be expressed by a $\Sigma_{n+1}$-formula with parameters $\alpha$, $\kappa$, $\theta$ and $\zeta$.
  Moreover, since $\kappa+1\subseteq M$ and $M$ is $\Pi_{n+1}(V_{\kappa +1})$-upwards correct, this statement holds in $M$.
  By the elementarity of $j$ and the fact that $\lambda'\in C^{(n+1)}$, we now know that, in  $V$, there exists an ordinal $\eta$ and an elementary embedding $i:V_\lambda\to V_\eta$ with $\crit{i}\in(\alpha,\mu]$, $i(\mu)\geq\lambda$ and $i(\nu)\in C^{(n)}$. 
  This shows that $\crit{i}\in (\alpha,\mu]$ is a $[\mu,\lambda)$-$C^{(n)}$-extendible cardinal. 
\end{proof}

\begin{corollary}\label{corollary:ESRyieldsLimitSupercompact}
 Let $\kappa$ be a cardinal and let $\alpha<\kappa$ be an ordinal. 
\begin{enumerate}
\item   If $\Pi_1(\{\alpha\})^{ic}$-$\ESR(\kappa,\lambda)$ holds for a proper class of cardinals $\lambda$, then the interval $(\alpha,\kappa]$ contains a  supercompact cardinal. 
  
  \item   For each natural number $n>0$, if $\Pi_{n+1}(\{\alpha\})^{ic}$-$\ESR(\kappa,\lambda)$ holds for a proper class of cardinals $\lambda$, then the interval $(\alpha,\kappa]$ contains a  $[\kappa ,\infty)$-$C^{(n)}$-extendible cardinal. \qed 
  \end{enumerate}
\end{corollary}

A combination of the above results now yield   short proofs  of Theorems \ref{theorem:LimitSupercompacts} and \ref{theorem:LimitCnExtendibles} stated in the Introduction.

\begin{proof}[Proof of Theorem \ref{theorem:LimitSupercompacts}]
 Let $\kappa$ be a singular cardinal. If $\kappa$ is a limit of supercompact cardinals, then Corollary \ref{corollary:ESRfromSRatSingulars} shows that  $\Sigma_2(V_\kappa)^{ic}$-$\ESR(\kappa,\lambda)$ holds for a proper class of cardinals $\lambda$.
 In the other direction, if $\Pi_1(\kappa)^{ic}$-$\ESR(\kappa,\lambda)$ holds for a proper class of cardinals $\lambda$, then we can apply Corollary \ref{corollary:ESRyieldsLimitSupercompact} to show that $\kappa$ is a limit of supercompact cardinals. 
\end{proof}

\begin{proof}[Proof of Theorem \ref{theorem:LimitCnExtendibles}]
 Let $\kappa$ be a singular cardinal and let $n>0$ be a natural number. If $\kappa$ is a limit of $[\kappa ,\infty)$-$C^{(n)}$-extendible cardinals, then Corollary \ref{corollary:ESRfromSRatSingulars} shows that  $\Sigma_{n+2}(V_\kappa)^{ic}$-$\ESR(\kappa,\lambda)$ holds for a proper class of cardinals $\lambda$. 
For the other direction, if $\Pi_{n+1}(V_\kappa)^{ic}$-$\ESR(\kappa,\lambda)$ holds for a proper class of cardinals $\lambda$, then  Corollary \ref{corollary:ESRyieldsLimitSupercompact} allows us to conclude  that $\kappa$ is a limit of $[\kappa,\infty)$-$C^{(n)}$-extendible cardinals. 
\end{proof}

We will eventually show that the strength of the principle $\Pi_1^{ic}$-$\ESR(\kappa)$ further increases significantly if $\kappa$ is a regular cardinal. 
  More specifically, we will  show  (see Corollary \ref{corollary:AlmostHugeFromPi1ic} below and Theorem  \ref{theorem:CharacterizationPI-ESR} stated in the Introduction)  that this assumption implies the existence of an almost huge cardinal. 
 The next lemma is the starting point of this analysis. It will  allow us to show that the least regular cardinal $\mu$  satisfying $\Pi_n^{ic}$-$\ESR(\mu)$ coincides with the least cardinal $\nu$  satisfying $\Pi_n(V_\nu)$-$\ESR(\nu)$.

\begin{lemma}\label{lemma:FullReflectionC__}
 Given a natural number $n>0$, assume that 
 \begin{itemize}
     \item $\kappa<\lambda<\lambda'$ are cardinals with $\beth_\lambda<\lambda'\in C^{(n)}$, 
     
     \item $M$ is a transitive, $\Pi_n(V_{\kappa+1})$-upwards correct set with $V_\kappa\cup\{\kappa\}\subseteq M$, and 
     
     \item $j:M\to H_{\lambda'}$ is an elementary embedding with   $j(\kappa)=\lambda$.  
 \end{itemize}
 Then the following statements hold: 
  \begin{enumerate}
      \item If $\mu\in \ran{j}\cap(\crit{j},\kappa]$ and $z\in M$ with $j(z)=z$, then $j(\mu)>\mu$ and $\Pi_n(\{z\})$-$\ESR(\mu,j(\mu))$ holds. 
      
      \item If $\mu=j(\crit{j})$, then $\mu\leq\kappa$ implies that  $\Pi_n(V_\mu)$-$\ESR(\mu,j(\mu))$ holds. 
  \end{enumerate}
\end{lemma}

\begin{proof}
 (i) Pick $\theta\in M$ with $j(\theta)=\mu$ and set $\nu=j(\mu)\leq\lambda$. 
  Since $j(\kappa)=\lambda>\kappa\geq\mu>\crit{j}$, we know that $\crit{j}\leq\theta$ and this allows us to apply Proposition \ref{proposition:HighCritical} to show that $\theta<\mu<\nu$.  
  Fix a $\Pi_n$-formula $\varphi(v_0,v_1)$ with the property that the class  $\Ce=\Set{A}{\varphi(A,z)}$ consists of structures of the same type. 
  Assume, towards a contradiction, that there exists $B\in\Ce$ of rank $\nu$ with the property that for all $A\in\Ce$ of rank $\mu$, there exists no elementary embedding of $A$ into $B$. 
  Since $\beth_\nu\leq\beth_\lambda<\lambda'$, we know that $B\in H_{\lambda'}$ and the fact that $\lambda'\in C^{(n)}$ implies that $\varphi(B,z)$ holds in $H_{\lambda'}$ and for every structure $A$ of the given type and rank $\mu$, either  $\varphi(A,z)$ fails in $H_{\lambda'}$ or $H_{\lambda'}$ contains no elementary embedding of $A$ into $B$. 
  The elementarity of $j$ then yields a structure $B_0\in M$ of the given type and rank $\mu$ such that $\varphi(B_0,z)$ holds in $M$ and for every structure $A$ of the given type and rank $\theta$, either $\varphi(A,z)$ fails in $M$ or $M$ contains no elementary embedding of $A$ into $B_0$. 
  Our setup then ensures that $B_0$ is an element of $\Ce$ of rank $\mu$ and the embedding $j$ gives rise to an elementary embedding $i$ of $B_0$ into $j(B_0)$. 
  But this yields a contradiction to the elementarity of $j$, because $B_0$ and $i$ are both contained in $H_{\lambda'}$, and the fact that $M$ is correct about the sentence $\varphi(B_0,z)$, and $\lambda^\prime \in C^{(n)}$, implies that $\varphi(B_0,u)$ holds in $H_{\lambda'}$. 
 %
 %
  %
  %

 (ii) Set $\nu=j(\mu)>\mu$ and fix a $\Pi_n$-formula $\varphi(v_0,v_1)$. 
  Assume, towards a contradiction, that there exists $z\in V_\mu$ with the property that  $\Ce=\Set{A}{\varphi(A,z)}$ is a class of structures of the same type and there exists $B\in\Ce$ of rank $\nu$ such that for all $A\in\Ce$ of rank $\mu$, there is no elementary embedding of $A$ into $B$. 
  Then the fact that $\nu\leq\lambda\leq\beth_\lambda<\lambda'\in C^{(n)}$ implies that   this statement also holds in $H_{\lambda'}$, and hence, in $M$, there exists $z_0\in V_{\crit{j}}$ with the property that  $\Ce_0=\Set{A}{\varphi(A,z_0)}$ is a class of structures of the same type and there exists $B\in\Ce_0$ of rank $\mu$ such that for all $A\in\Ce_0$ of rank $\crit{j}$, there is no elementary embedding of $A$ into $B$. 
  Since $M$ is $\Pi_n(V_{\kappa +1})$-upwards correct, $\lambda'\in C^{(n)}$ and $j(z_0)=z_0$, we know that $\varphi(B,z)$ and $\varphi(j(B),z_0)$ hold in $H_{\lambda'}$. So, since  $\mu\in(\crit{j},\kappa]$, we can now proceed as in the proof of (i) to derive a contradiction. 
\end{proof}

\begin{lemma}\label{lemma:NonIsoFromIso}
 Given a natural number $n>0$, let $\mu$ be a regular cardinal with the property that $\Pi_n^{ic}$-$\ESR(\mu,\lambda)$ holds for some cardinal $\lambda>\mu$. 
 Then there exists an inaccessible cardinal $\delta\leq\mu$ with the property that $\Pi_n(V_\delta)$-$\ESR(\delta,\rho)$ holds for an inaccessible cardinal $\rho$ with $\delta<\rho\leq\lambda$. 
\end{lemma}

\begin{proof}
 Pick $\lambda'\in C^{(n)}$ greater than $\beth_\lambda$ and use Lemma \ref{lemma:EmbFromESRisoC}  to find a cardinal $\kappa\leq\mu$ with $\mu=\cof{\mu}\leq\beth_\kappa$, a transitive, $\Pi_n(V_{\kappa+1})$-upwards correct set $M$ with $V_\kappa\cup\{\beth_\kappa\}\subseteq M$ and an elementary embedding $j:M\to H_{\lambda'}$ with $\crit{j}<\kappa$, $j(\crit{j})\leq\mu$, $\mu\in\ran{j}$ and $j(\kappa)=\lambda$. 
 Since $V_\kappa\subseteq M$, we have that $\crit{j}$ is an  inaccessible cardinal.

   \begin{claim*}
    $j(\crit{j})\leq\kappa$. 
   \end{claim*}
   
   \begin{proof}[Proof of the Claim]
    Let $\varepsilon$ be minimal with $\beth_\varepsilon\geq\mu$. Note that $\varepsilon \leq \kappa$. 
    Since $\mu$ is an element of $\ran{j}$, and $\varepsilon$ is definable from $\mu$, we can find $\zeta\in M$ with $j(\zeta)=\varepsilon$. 
    Moreover, since $\crit{j}$ is an inaccessible cardinal smaller than $\beth_\varepsilon$, we know that $\zeta\geq\crit{j}$. 
    But this allows us to conclude that  $j(\crit{j})\leq j(\zeta)=\varepsilon\leq\kappa$. 
   \end{proof}

  Define $\delta=j(\crit{j})\in M$ and $\rho=j(\delta)$. Then elementarity implies that $\delta$  and $\rho$ are also  inaccessible cardinals. 
  Moreover, an application of the second part of Lemma \ref{lemma:FullReflectionC__} directly shows that $\Pi_n(V_\delta)$-$\ESR(\delta,\rho)$ holds. 
\end{proof}


\section{The $\Pi_n$-case for arbitrary classes}

In order to study principles of exact structural reflection for $\Pi_n$-definable classes of structures that are not necessarily closed under isomorphic copies, we analyse connections between the validity of these principles and the existence of weakly $n$-exact cardinals. 
 The following variation of Definition \ref{definition:WeaklyExactCardinalUp} will allow us to state our results more precisely:

\begin{definition}
 Given cardinals $\kappa<\lambda$, a set $z\in V_\kappa$ and a natural number $n>0$, we say that $\kappa$ is \emph{weakly $n$-exact for $\lambda$ and $z$} if for  every $A\in V_{\lambda +1}$, there exists 
     a transitive, $\Pi_n(V_{\kappa+1})$-upwards correct set $M$ with $V_\kappa\cup \{\kappa\} \subseteq M$,  a cardinal $\lambda'\in C^{(n-1)}$ greater than $\beth_\lambda$ 
     and an elementary embedding $j:M\to H_{\lambda'}$ with $j(\kappa)=\lambda$,  $j(z)=z$ and $A\in\ran{j}$. 
\end{definition}

\begin{proposition}\label{proposition:ParametersFromParametrically}
 If a cardinal $\kappa$ is weakly parametrically $n$-exact for some cardinal $\lambda$, then $\kappa$ is weakly $n$-exact for $\lambda$ and all $z\in V_\kappa$. 
\end{proposition}

\begin{proof}
 Given $z\in V_\kappa$ and $A\in V_{\lambda+1}$, our assumptions yield a transitive, $\Pi_n(V_{\kappa+1})$-upwards correct set $M$ with $V_\kappa\cup \{\kappa\} \subseteq M$,  a cardinal $\lambda'\in C^{(n-1)}$ greater than $\beth_\lambda$ and an elementary embedding $j:M\to H_{\lambda'}$ with $j(\kappa)=\lambda$, $j(\crit{j})=\kappa$ and $A,z\in\ran{j}$. 
 Pick $z_0\in M$ with $j(z_0)=z$. Since we have $j(\rank{z_0})=\rank{z}<\kappa=j(\crit{j})$, it follows that $z_0\in M\cap V_{\crit{j}}$ and hence $z_0=j(z_0)=z$.  
\end{proof}

\begin{proposition}\label{proposition:ESRFromExact}
  If  $\kappa$ is weakly $n$-exact for $\lambda$ and $z$, then $\Pi_{n}(\{z\})$-$\ESR(\kappa ,\lambda)$  holds. 
  In particular, if $\kappa$ is weakly parametrically $n$-exact for $\lambda$, then $\Pi_{n}(V_\kappa)$-$\ESR(\kappa ,\lambda)$   holds. 
\end{proposition}

\begin{proof}
 Fix a $\Pi_n$-formula $\varphi(v_0,v_1)$ with the property that the class $\Ce=\Set{A}{\varphi(A,z)}$ consists of structures of the same type and $B\in\Ce$ of rank $\lambda$. 
  By our assumptions,  there exists a transitive, $\Pi_n(V_{\kappa+1})$-upwards correct set $M$ with $V_\kappa\cup \{\kappa\} \subseteq M$, a cardinal $\lambda'\in C^{(n-1)}$ greater than $\beth_\lambda$ and an elementary embedding $j:M\to H_{\lambda'}$ with $j(\kappa)=\lambda$,  $j(z)=z$ and $B\in\ran{j}$. 
 Pick $A\in M$ with $j(A)=B$. 
 Since all $\Pi_n$-formulas with parameters in $H_{\lambda'}$ are downwards absolute from $V$ to $H_{\lambda'}$, we know that $\varphi(B,z)$ holds in $H_{\lambda'}$ and the structure $B$ has rank $\lambda$ in $H_{\lambda'}$. 
 But this means that, in $M$, the statement  $\varphi(A,z)$ holds and $A$ has rank $\kappa$.
 The fact that $M$ is $\Pi_n(V_{\kappa+1})$-upwards correct implies that the statement $\varphi(A,z)$ holds in $V$ and hence we can conclude that $A$ is a structure in $\Ce$ of rank $\kappa$. Moreover, we know that the embedding $j$ induces an elementary embedding of $A$ into $B$. 
 The second part of the proposition follows directly from Proposition \ref{proposition:ParametersFromParametrically}. 
\end{proof}

For each natural number $n>0$ and every set $z$, let $\mathcal{E}_n(z)$ denote the  class of structures  $\langle D, E, a,b,c \rangle$  (in the language of set theory with three additional constant symbols) with the property that $\rank{D}\subseteq D$, $E$ is a  well-founded and extensional relation on $D$ and,  if $\pi: \langle D,E\rangle \to \langle M,\in\rangle$ is the corresponding transitive collapse, then $V_{\rank{D}}\cup\{\rank{D}\}\subseteq M$, $M$ is $\Pi_n(V_{\rank{D}+1})$-upwards correct,  $\pi(b)=\rank{D}$, $\pi(c)=z$ and $\pi^{{-}1}\restriction\rank{D}=\id_{\rank{D}}$. 
 Note that the class $\mathcal{E}_n(z)$ is definable by a $\Pi_{n}$-formula with parameter $z$.

\begin{lemma}\label{lemma:WeakExactFromESR}
 Let $\kappa<\lambda$ be cardinals, let $z\in V_\kappa$ and let $n>0$ be a natural number with the property that $\ESR_{\mathcal{E}_n(z)}(\kappa,\lambda)$ holds. 
 Given $B\in V_{\lambda +1}$ and a cardinal $\lambda'\in C^{(n)}$ greater than $\lambda,$\footnote{Since the class $C^{(1)}$ consists of all cardinals $\mu$ with the property that $H_\mu=V_\mu$, the given assumption ensures that $\lambda'$ is greater than $\beth_\lambda$.} there exists a transitive, $\Pi_n(V_{\kappa+1})$-upwards correct set $M$ with $V_\kappa\cup\{\kappa\}\subseteq M$ and an elementary embedding $j:M\to H_{\lambda'}$ with $j(\kappa)=\lambda$, $j(z)=z$ and $B\in\ran{j}$. 
  In particular, the cardinal $\kappa$ is weakly $n$-exact for $\lambda$ and $z$. 
\end{lemma}

\begin{proof}
 Let $X$ be an elementary submodel of $H_{\lambda'}$  of cardinality $\beth_\lambda$ with  $V_\lambda \cup \{ \lambda , B\}\subseteq Y$.  
 Pick a bijection $f:X\to V_\lambda$ with $f\restriction\lambda=\id_\lambda$ and let $R$ be the induced binary relation on $V_\lambda$. 
 Since the transitive collapse of $\langle V_\lambda,R\rangle$ is the composition of $f^{{-}1}$ and the transitive collapse of $\langle X,\in\rangle$ and all $\Sigma_n$-formulas using parameters from the transitive part of $X$ are absolute between $V$ and the transitive collapse of $X$, it follows that the structure $\langle V_\lambda,R,f(B),f(\lambda),f(z)\rangle$ is an element of $\mathcal{E}_n(z)$ of rank $\lambda$. 
 
 By our assumptions, we can find a  structure $\langle D,E,a,b,c\rangle$ of rank $\kappa$ in $\mathcal{E}_n(z)$ such that there exists an elementary embedding $$i:\langle D,E,a,b,c\rangle\to\langle V_\lambda,R,f(B),f(\lambda),f(z)\rangle.$$
 Let $\pi:\langle D,E\rangle\to\langle M,\in\rangle$ denote the corresponding transitive collapse. Set $A=\pi(a)$ and $$j ~ = ~ f^{{-}1}\circ i\circ\pi^{{-}1}:M\to H_{\lambda'}.$$ 
  Then $M$ is a transitive set with $V_\kappa\cup\{\kappa\}\subseteq M$, $A\in M\cap V_{\kappa+1}$ with the property that $M$ is $\Pi_n(V_{\kappa+1})$-upwards correct and $j$ is an elementary embedding with $j(\kappa)=\lambda$, $j(A)=B$ and $j(z)=z$. 
\end{proof}

A combination of Proposition \ref{proposition:ESRFromExact} and Lemma \ref{lemma:WeakExactFromESR} now yields the following equivalence:

\begin{corollary}\label{corollary:WeakExactStronger}
 The following statements are equivalent for all natural numbers $n>0$, all cardinals $\kappa<\lambda$ and all $z\in V_\kappa$: 
 \begin{enumerate}
     \item $\Pi_n(\{z\})$-$\ESR(\kappa,\lambda)$. 
     
     \item $\kappa$ is weakly  $n$-exact for $\lambda$ and $z$. 
     
     \item For all $A\in V_{\lambda +1}$ and $\lambda<\lambda'\in C^{(n)}$, there exists a transitive, $\Pi_n(V_{\kappa+1})$-upwards correct set $M$ with $V_\kappa\cup\{\kappa\}\subseteq M$ and an elementary embedding $j:M\to H_{\lambda'}$ satisfying  $j(\kappa)=\lambda$, $j(z)=z$ and $A\in\ran{j}$. \qed 
 \end{enumerate}
\end{corollary}

Arguments contained in the proofs of the above results also allow us to prove the following parametrical version of Corollary \ref{corollary:WeakExactStronger} that will be needed later on.

\begin{lemma}\label{lemma:WeakParaExactStronger}
 Let $n>0$ be a natural number, and let $\kappa$ be weakly parametrically $n$-exact for some cardinal $\lambda>\kappa$. If $\lambda<\lambda'\in C^{(n)}$ and $B\in V_{\lambda+1}$, then there exists a transitive, $\Pi_n(V_{\kappa+1})$-upwards correct set $M$ with $V_\kappa\cup\{\kappa\}\subseteq M$ and an elementary embedding $j:M\to H_{\lambda'}$ satisfying $j(\crit{j})=\kappa$,   $j(\kappa)=\lambda$ and $B\in\ran{j}$.
\end{lemma}

\begin{proof}
 Pick  an elementary submodel $X$ of $H_{\lambda'}$  of cardinality $\lambda$ with  $V_\lambda \cup \{\lambda, B\}\subseteq X$ and a bijection $f:X\to V_\lambda$ with $f\restriction\lambda=\id_\lambda$. 
 Let  $\langle V_\lambda,R,f(B),f(\lambda),\emptyset\rangle$ denote the corresponding  structure in  $\mathcal{E}_n(\emptyset)$ of rank $\lambda$ constructed in the proof of Lemma \ref{lemma:WeakExactFromESR}. 
 By our assumptions, we can find a transitive, $\Pi_n(V_{\kappa+1})$-upwards correct set $N$ with $V_\kappa\cup\{\kappa\}\subseteq N$, a cardinal $\beth_\lambda<\eta\in C^{(n-1)}$ and an elementary embedding $i:N\to H_\eta$ with $i(\crit{i})=\kappa$, $i(\kappa)=\lambda$ and $R,f(B),f(\lambda)\in\ran{i}$. 
 Pick $R_0,a_0,b_0\in N$ with $i(R_0)=R$, $i(a_0)=f(B)$ and  $i(b_0)=f(\lambda)$. 
 Then the elementarity of $i$, the $\Pi_n(V_{\kappa+1})$-upwards correctness of $N$ and the fact that $V_\lambda$ is contained in $H_\eta$ ensure that $\langle V_\kappa,R_0,a_0,b_0,\emptyset\rangle$ is a structure in $\mathcal{E}_n(\emptyset)$ of rank $\kappa$ and $$i\restriction V_\kappa:\langle V_\kappa,R_0,a_0,b_0,\emptyset\rangle\to\langle V_\lambda,R,f(B),f(\lambda),\emptyset\rangle$$ is an elementary embedding. 
 Let $\pi:\langle V_\kappa,R_0\rangle\to\langle M,\in\rangle$ denote the corresponding transitive collapse. Then $M$ is a  $\Pi_n(V_{\kappa+1})$-upwards correct set with $V_\kappa\cup\{\kappa\}\subseteq M$,  $\pi^{{-}1}\restriction\kappa=\id_\kappa$ and $\pi(b_0)=\kappa$. 
 If we now define $$j ~ = ~ f^{{-}1}\circ i\circ \pi^{{-}1}:M\to H_{\lambda'},$$ then $j$ is a non-trivial  elementary embedding between transitive structures with $j\restriction\crit{i}=\id_{\crit{i}}$, $j(\crit{i})=\kappa>\crit{i}=\crit{j}$, $j(\kappa)=\lambda$ and $B\in\ran{j}$. 
\end{proof}

%

In the remainder of this section, we prove that for all natural numbers $n>0$, exact structural reflection for $\Pi_n$-definable classes at cardinals in $C^{(n)}$ is strictly stronger than exact structural reflection for $\Sigma_n$-definable classes at such cardinals. This will follow from Corollary \ref{corollary:WeakExactStronger} and the next Lemma.

\begin{lemma}
 Let $n>0$ be a natural number and let $\kappa\in C^{(n)}$ be weakly parametrically $n$-exact for some cardinal $\lambda>\kappa$. 
 Then the set of cardinals $\mu<\kappa$ such that $\Sigma_n(V_\mu)$-$\ESR(\mu,\kappa)$ holds is stationary in $\kappa$. 
 %
\end{lemma}

\begin{proof}
 Fix a closed unbounded subset $K$ of $\kappa$. 
 By Lemma \ref{lemma:WeakParaExactStronger}, there exists a transitive, $\Pi_n(V_{\kappa+1})$-upwards correct set $M$ with $V_\kappa\cup\{\kappa\}\subseteq M$, cardinals $\kappa<\lambda<\lambda'\in C^{(n)}$ and an elementary embedding $j:M\to H_{\lambda'}$ with $j(\crit{j})=\kappa$, $j(\kappa)=\lambda$ and $K\in\ran{j}$. 
  Then $\crit{j}\in K$. 
  Assume, towards a contradiction, that $\Sigma_n(V_{\crit{j}})$-$\ESR(\crit{j},\kappa)$ fails. 
   Then there exists a $\Sigma_n$-formula $\varphi(v_0,v_1)$, $z\in V_{\crit{j}}$ and a structure $B$ of rank $\kappa$ with the property that the class $\Ce=\Set{A}{\varphi(A,z)}$ consists of structures of the same type, $\varphi(B,z)$ holds and for all $A\in\Ce$ of rank $\crit{j}$, there is no elementary embedding of $A$ into $B$. 
  Since $\kappa$ is inaccessible and therefore every elementary embedding of a structure of rank less than $\kappa$ into a structure of rank $\kappa$ is an element of $V_\kappa$, and moreover $\kappa\in C^{(n)}$, there is a $\Sigma_n$-formula with parameters in $V_\kappa\cup\{V_\kappa\}$ that states that there exists a structure $B$ of rank $\kappa$ with the property that $\varphi(B,z)$ holds and for every structure $A\in V_\kappa$ of rank $\crit{j}$ such that $\varphi(A,z)$ holds in $V_\kappa$, there is no   elementary embedding from $A$ into $B$ in $V_\kappa$.  
  Since $M$ is $\Pi_n(V_{\kappa+1})$-upwards correct and $V_\kappa\cup\{V_\kappa\}\subseteq M$, this yields a structure $B_0$ of rank $\kappa$ in $M$ with the property that $\varphi(B_0,z)$ holds in $M$ and for every structure  $A$ of rank $\crit{j}$ such that $\varphi(A,z)$ holds in $V_\kappa$, there is no elementary embedding of $A$ into $B_0$ in $V_\kappa$. 
  Moreover, since $M$ is $\Pi_n(V_{\kappa+1})$-upwards correct and $\kappa$ is an inaccessible cardinal in $C^{(n)}$, it follows that for every structure  $A$ of rank $\crit{j}$ such that $\varphi(A,z)$ holds in $M$, there is no elementary embedding of $A$ into $B_0$ in $M$. 
  Then, in $H_{\lambda'}$, for every structure  $A$ of rank $\kappa$ such that $\varphi(A,z)$ holds $H_{\lambda'}$, there is no elementary embedding of $A$ into $j(B_0)$ in $V_\lambda$. 
  Since $\kappa<\lambda<\lambda'\in C^{(n)}$, it now follows that $\varphi(B_0,z)$ holds in $H_{\lambda'}$ and the map $j\restriction B_0:B_0\to j(B_0)$ is an element of $H_{\lambda'}$. 
  But this yields a contradiction, because $j\restriction B_0$ is an elementary embedding of $B_0$ into $j(B_0)$ in $H_{\lambda'}$.  
\end{proof}


\section{The $\Sigma_{n+1}$-case}

Analogously to the theory developed in the previous section, we now analyse the relationship between the principle $\Sigma_{n+1}$-$\ESR$ and $n$-exact cardinals. First, let us consider the following variation of Definition \ref{defnexact}.

\begin{definition}
 Given cardinals $\kappa<\lambda$, a set $z\in V_\kappa$ and $n<\omega$, the cardinal $\kappa$ is \emph{$n$-exact for $\lambda$ and $z$} if for every $A\in V_{\lambda +1}$, there exists a cardinal   $\kappa'\in C^{(n)}$ greater than $\beth_\kappa$, a cardinal $\lambda'\in C^{(n+1)}$ greater than $\lambda$, an elementary submodel $X$ of $H_{\kappa'}$ with  $V_\kappa \cup \{\kappa\}   \subseteq X$, and an elementary embedding $j:X\to H_{\lambda'}$ with  $j(\kappa)=\lambda$, $j(z)=z$ and   $A\in\ran{j}$. 
\end{definition}

\begin{proposition}\label{proposition:ExactForZ}
 If a cardinal $\kappa$ is parametrically $n$-exact for some cardinal $\lambda$ (see Definition \ref{defnexact}), then $\kappa$ is  $n$-exact for $\lambda$ and all $z\in V_\kappa$. \qed 
\end{proposition}

\begin{proposition}\label{proposition:ESRfromExact}
  If  $\kappa$ is  $n$-exact for $\lambda$ and $z$, then $\Sigma_{n+1}(\{z\})$-$\ESR(\kappa ,\lambda)$   holds. In particular, if $\kappa$ is  parametrically $n$-exact for $\lambda$, then $\Sigma_{n+1}(V_\kappa)$-$\ESR(\kappa ,\lambda)$   holds. 
\end{proposition}

\begin{proof}
 Pick a $\Sigma_{n+1}$-formula $\varphi(v_0,v_1)$ with the property that the class $\Ce=\Set{A}{\varphi(A,z)}$ consists of structures of the same type and fix $B\in\Ce$ of rank $\lambda$. 
 By our assumptions, there exists a cardinal $\beth_\kappa<\kappa'\in C^{(n)}$, a cardinal $\lambda<\lambda'\in C^{(n+1)}$, an elementary submodel $X$ of $H_{\kappa'}$ with $V_\kappa\cup\{\kappa\}\subseteq X$ and an elementary embedding $j:X\to H_{\lambda'}$ with $j(\kappa)=\lambda$, $j(z)=z$ and $B\in\ran{j}$. 
 Pick $A\in X$ with $j(A)=B$. Then our setup ensures that $\varphi(A,z)$ holds and hence $A$ is a structure in $\Ce$ of rank $\kappa$. Moreover, the map $j$ induces an elementary embedding of $A$ into $B$. 
 These computations yield the first part of the proposition. The second part follows directly from a combination of the first part and Proposition \ref{proposition:ExactForZ}.
\end{proof}

 For each natural number $n>0$ and every set $z$, we let $\De_n(z)$ denote the  class of structures  $\langle D, E, a,b,c \rangle$  (in the language of set theory with three additional constant symbols)
 with the property that for some cardinal $\theta \in C^{(n)}$ greater than $\beth_{\rank{D}}$, 
  there exists an elementary submodel $X$ of $H_\theta$ with $V_{\rank{D}}\cup\{\rank{D}\}\subseteq X$
  and an isomorphism   $\pi:\langle D,E\rangle \to \langle X,\in\rangle$  with $\pi(b)=z$ and $\pi(c)=\rank{D}$.
 Note that the class $\De_n(z)$ is definable by a $\Sigma_{n+1}$-formula with parameter $z$.

 \begin{lemma}\label{lemma:ExactFromESR_}
 Let $\kappa<\lambda$ be cardinals, let $z\in V_\kappa$ and let $n>0$ be a natural number with the property that $\ESR_{\De_n(z)}(\kappa,\lambda)$ holds. 
 Then $\kappa$ is  $n$-exact for $\lambda$ and $z$. 
\end{lemma}
 
\begin{proof}
 Fix $A\in V_{\lambda+1}$, $\lambda'\in C^{(n+1)}$ greater than $\lambda$ and an elementary submodel $Y$ of $H_{\lambda'}$ of cardinality $\beth_\lambda$ with $V_\lambda\cup\{A,\lambda\}\subseteq Y$. 
 Pick a bijection $f:Y\to V_\lambda$ and let $R$ be the binary relation on $V_\lambda$ induced by $f$ and $\in$. 
  Then $\lambda'$ and $f^{{-}1}$ witness that $\langle V_\lambda,R,f(A),f(z),f(\lambda)\rangle$ is a structure of rank $\lambda$ in $\De_n(z)$. 
  By our assumptions, there exists a structure $\langle D,E,a,b,c\rangle$ of rank $\kappa$ in $\De_n(z)$ and an elementary embedding $$i:\langle D,E,a,b,c\rangle\to\langle V_\lambda,R,f(A),f(z),f(\lambda)\rangle.$$
  Pick a cardinal $\kappa'\in C^{(n)}$, an elementary submodel $X$ of $H_{\kappa'}$ and an isomorphism  $\pi:\langle D,E\rangle\to\langle X,\in\rangle$ witnessing that $\langle D,E,a,b,c\rangle$ is an element of $\De_n(z)$. 
   Then $\pi(b)=z$, $\pi(c)=\kappa$, $\beth_\kappa<\kappa'$ and  $V_\kappa\cup\{\kappa\}\subseteq X$. Define $$j ~ = ~ f^{{-}1}\circ i\circ\pi^{{-}1}:X\to H_{\lambda'}.$$ 
   Then $j$ is an elementary embedding with $j(\kappa)=\lambda$, $j(z)=z$ and $A=j(\pi(a))\in \ran{j}$. 
\end{proof}


 \section{Proofs of the main theorems}
 
 We shall now give a proof of Theorems \ref{theorem:CharacterizationPI-ESR} and \ref{theorem:CharacterizationSigma-ESR}.
 
 \begin{proof}[Proof of Theorem \ref{theorem:CharacterizationPI-ESR}]
 Fix a cardinal $\kappa$ and a natural number $n>0$. 
 
 The next two claims, together with  Proposition \ref{proposition:ESRFromExact} and Corollary \ref{corollary:WeakExactStronger}, will  allow us to conclude  that all  statements (i)-(v)  listed in the theorem are equivalent.
 
 \begin{claim*}
   If $\kappa$ is the least regular cardinal with the property that $\Pi_n^{ic}$-$\ESR(\kappa)$ holds, then $\kappa$ is weakly parametrically $n$-exact for some $\lambda>\kappa$. 
 \end{claim*}
 
 \begin{proof}[Proof of the Claim]
  By Lemma \ref{lemma:NonIsoFromIso} and the minimality of $\kappa$, we know that $\kappa$ is an inaccessible cardinal with the property that the principle  $\Pi_n(V_\kappa)$-$\ESR(\kappa,\lambda)$ holds for some inaccessible cardinal $\lambda>\kappa$. 
  Assume, towards a contradiction, that $\kappa$ is not weakly parametrically $n$-exact for $\lambda$. 
  Then there exists $A\in V_{\lambda+1}$ with the property that $j(\crit{j})\neq\kappa$ holds whenever $M$ is a transitive, $\Pi_n(V_{\kappa+1})$-upwards correct  set  with $V_\kappa\cup \{\kappa\} \subseteq M$, $\lambda'\in C^{(n-1)}$ is a cardinal  greater than $\lambda$      and $j:M\to H_{\lambda'}$ is an elementary embedding  with $j(\kappa)=\lambda$ and $A\in\ran{j}$.  
  An application of Lemma \ref{lemma:WeakExactFromESR} now allows us to find a transitive, $\Pi_n(V_{\kappa+1})$-upwards correct  set $M$ with $V_\kappa\cup \{\kappa\} \subseteq M$, a cardinal $\lambda'$ with $\lambda<\lambda'\in C^{(n)}$ and an elementary embedding $j:M\to H_{\lambda'}$ with $j(\kappa)=\lambda$ and $A,\kappa\in\ran{j}$. 
  Then $j(\crit{j})\leq\kappa$ and hence  $j(\crit{j})<\kappa$. Set $\mu=j(\crit{j})\in M\cap\kappa$ and $\nu=j(\mu)$. 
  We can now apply Lemma \ref{lemma:FullReflectionC__} to conclude that $\Pi_n(V_\mu)$-$\ESR(\mu,\nu)$ holds, contradicting the minimality of $\kappa$. 
 \end{proof}

 \begin{claim*}
  If $\kappa$ is the least cardinal that is weakly $n$-exact for some cardinal $\lambda$, then $\kappa$ is regular. 
 \end{claim*}
 
 \begin{proof}[Proof of the Claim]
  Assume, towards a contradiction, that $\kappa$ is singular.  
  Then Proposition \ref{proposition:ESRFromExact} yields a cardinal $\lambda>\kappa$ with the property that $\Pi_n$-$\ESR(\kappa,\lambda)$ holds. 
  In this situation, we can apply Lemma \ref{lemma:WeakExactFromESR} to find a cardinal $\lambda'$ with  $\beth_\lambda<\lambda'\in C^{(n)}$, a transitive, $\Pi_n(V_{\kappa+1})$-upwards correct set $M$ with $V_\kappa\cup\{\kappa\} \subseteq M$ and an elementary embedding $j:M\to H_{\lambda'}$ with $j(\kappa)=\lambda$ and $\kappa\in\ran{j}$. 
  Set $\mu=j(\crit{j})$. Then elementarity implies that $\mu$ is regular and, since the fact that $\kappa\in j(\kappa)\cap\ran{j}$ ensures that   $\mu\leq\kappa$, we know that   $\mu<\kappa$. 
  In this situation, the second part of Lemma \ref{lemma:FullReflectionC__} implies that $\Pi_n$-$\ESR(\mu,j(\mu))$ holds and this allows us to apply Corollary \ref{corollary:WeakExactStronger} to conclude that $\mu$ is weakly $n$-exact for $j(\mu)$,   contradicting the minimality of $\kappa$.  
 \end{proof}

    By the first claim above,  the cardinal that satisfies (i) of the Theorem is greater than or equal to the cardinal that satisfies (v), which, by  Proposition \ref{proposition:ESRFromExact},  
    is greater than or equal to the cardinal that satisfies (iii).  
    Moreover, the cardinal satisfying (iii) is obviously greater than or equal to the cardinal that satisfies (ii), 
    and an application of Corollary \ref{corollary:WeakExactStronger} then shows that the cardinal satisfying (ii) is greater than or equal to the cardinal that satisfies (iv). 
    Our second claim then shows that the cardinal satisfying (iv) is regular and this allows us to use Corollary \ref{corollary:WeakExactStronger} again to conclude that it is greater than or equal to the cardinal satisfying (i). 
    This shows that all of these cardinals are equal. 
     \end{proof}

\medskip
 
 \begin{proof}[Proof of Theorem \ref{theorem:CharacterizationSigma-ESR}]
  Fix a cardinal $\kappa$ and a natural number $n>0$. 
  
  The next two claims, together with  Proposition \ref{proposition:ESRfromExact}, will  allow us to conclude  that all  statements (i)-(iv)  listed in the theorem are equivalent.
 
 \begin{claim*}
  If $\kappa$ is the least  cardinal with the property that $\Sigma_{n+1}$-$\ESR(\kappa)$ holds, then $\kappa$ is parametrically $n$-exact for some $\lambda>\kappa$. 
 \end{claim*}
 
 \begin{proof}[Proof of the Claim]
  Fix $D\in V_{\lambda+1}$. 
  Using Lemma \ref{lemma:ExactFromESR_}, we find a cardinal $\beth_\kappa<\kappa'\in C^{(n)}$, cardinals $\kappa<\lambda<\lambda'\in C^{(n+1)}$, an elementary submodel $X$ of $H_{\kappa'}$ with $V_\kappa\cup\{\kappa\}\subseteq X$ and an elementary embedding $j:X\to V_{\lambda'}$ with $j(\kappa)=\lambda$ and $D,\kappa\in\ran{j}$. 
  Set $\mu=j(\crit{j})\leq\kappa\in X$ and $\nu=j(\mu)$.   Assume, towards a contradiction, that $\Sigma_{n+1}$-$\ESR(\mu,\nu)$ fails. 
  Then we can find a $\Sigma_{n+1}$-formula $\varphi$ with the property that the class  $\Ce=\Set{A}{\varphi(A)}$ consists of structures of the same type and $B\in\Ce$ of rank $\nu$ with the property that for all $A\in \Ce$ of rank $\mu$, there is no elementary embedding of $A$ into $B$. 
  Since $\lambda'\in C^{(n+1)}$, this statement also holds in $H_{\lambda'}$ and therefore the elementarity of $j$ allows us to find $B_0\in X$ of rank $\mu$ with the property that $\varphi(j(B_0))$ holds in $H_{\lambda'}$ and for all $A\in V_{\mu+1}\setminus V_\mu$ such that $\varphi(A)$ holds, there is no elementary embedding of $A$ into $j(B_0)$ in $H_{\lambda'}$. 
  But this yields a contradiction, because our setup ensures that $\varphi(B_0)$ holds in $V$ and the fact that $V_\kappa\subseteq X$ implies that $j$ induces an elementary embedding of $B_0$ into $j(B_0)$ that is an element of $H_{\lambda'}$. 
  In this situation, the minimality of $\kappa$ implies that $\kappa=\mu$ and $\nu=\lambda$. In particular, we can conclude that $\kappa$ is parametrically $n$-exact for $\lambda$. 
 \end{proof}

 Since Proposition \ref{proposition:ESRfromExact} shows that $\Sigma_{n+1}(V_\kappa)$-$\ESR(\kappa)$ holds whenever $\kappa$ is parametrically $n$-exact for some cardinal $\lambda>\kappa$ and therefore  all statements listed in the theorem imply that $\Sigma_{n+1}$-$\ESR(\kappa)$ holds, the above claim allows us to conclude that all of the listed statements are equivalent. 
 \end{proof}
 
 \medskip
 
 In the remainder of this section, we  show that, for all $n>0$, exact structural reflection for $\Sigma_{n+1}$-definable classes is strictly stronger than exact structural reflection for $\Pi_n$-definable classes.  
  In combination with the equivalences provided by Theorems \ref{theorem:CharacterizationPI-ESR} and \ref{theorem:CharacterizationSigma-ESR}, 
  the following lemma shows that, in general, the principle $\Pi_n(V_\kappa)$-$\ESR(\kappa)$ does not imply the principle $\Sigma_{n+1}$-$\ESR(\kappa)$. 
  This statement should be compared with the results of {\cite[Section 4]{Ba:CC}}, showing that the validity of the principle $\mathrm{SR}$ for $\Pi_n$-definable classes of structures is equivalent to the validity of this principle for $\Sigma_{n+1}$-definable classes.

\begin{lemma}\label{lemma:CharESREmbSigma2}
 Let $n>0$ be a natural number and let $\kappa$ be parametrically $n$-exact for some cardinal $\lambda>\kappa$. 
 Then the set of cardinals $\mu<\kappa$ with  the property that $\Pi_n(V_\mu)$-$\ESR(\mu,\kappa)$ holds is stationary in $\kappa$. 
 In particular, there exists a cardinal $\mu<\kappa$  that is weakly parametrically $n$-exact for some cardinal $\nu>\mu$. 
\end{lemma}

\begin{proof}
 Fix a closed unbounded subset $K$ of $\kappa$. 
 Using our assumptions, we can find a cardinal  $\kappa<\kappa'\in C^{(n)}$, a cardinal $\lambda<\lambda'\in C^{(n+1)}$, an elementary submodel $X$ of $H_{\kappa'}$ with $V_\kappa\cup\{\kappa\}\subseteq X$ and an elementary embedding $j:X\to H_{\lambda'}$ with $j(\kappa)=\lambda$, $j(\crit{j})=\kappa$ and $K\in\ran{j}$. 
 Then $\crit{j}$ is an element of $K$. 
 
 Now, assume towards a contradiction, that there is a $\Pi_n$-formula $\varphi(v_0,v_1)$ and $z\in V_\crit{j}$ such that the class $\Ce=\Set{A}{\varphi(A,z)}$ consists of structures of the same type and there exists $B\in\Ce$ of rank $\kappa$ such that for all $A\in\Ce$ of rank $\crit{j}$, there is no elementary embedding from $A$ into $B$. 
 Since $\kappa<\kappa'\in C^{(n)}$, these statements hold in $H_{\kappa'}$ and we can find $B\in\Ce\cap X$ of rank $\kappa$ with the property that  for every $A\in\Ce$ of rank $\crit{j}$, there is no elementary embedding of $A$ into $B$. 
 In this situation, since $\lambda<\lambda'\in C^{(n+1)}$, elementarity implies that $j(B)$ is an element of $\Ce$ and for every $A\in\Ce$ of rank $\kappa$, there is no elementary embedding  of $A$ into $j(B)$. 
 But this yields a contradiction, because the fact that $V_\kappa$ is a subset of $X$ implies that $j$ induces an elementary embedding of $B$ into $j(B)$. 
 
 The above computations yield the first part of the lemma. The second part follows directly from a combination of the first part with Theorem \ref{theorem:CharacterizationPI-ESR}. 
\end{proof}


\section{The strength of exact cardinals}\label{section:strength}

In this section, we measure the strength of the principles of exact structural reflection introduced above by positioning exact and weakly exact cardinals in the  hierarchy of large cardinals. 
 We start by deriving lower bounds for their consistency strength by showing that the existence of such cardinals implies the existence of many \emph{almost huge} cardinals below them. 
 
 Recall that a cardinal $\kappa$ is \emph{almost huge} if there exists a transitive class $M$ and a non-trivial elementary embedding $j:V\to M$ with  $\crit{j}=\kappa$ and ${}^{{<}j(\kappa)}M\subseteq M$.
We then say that a 
cardinal $\kappa$ is almost huge  with \emph{target $\lambda$} if there exists an embedding $j$ witnessing the hugeness of $\kappa$ with $j(\kappa)=\lambda$. 

 The following standard argument will allow us to prove these implications:

\begin{lemma}\label{lemma:AlmostHugeFromEmb}
 Let $\kappa<\lambda$ be cardinals with the property that there exists a non-trivial elementary embedding $j:V_\kappa\to V_\lambda$ with  $j(\crit{j})=\kappa$. Then $\crit{j}$ is almost huge with target $\kappa$.  
\end{lemma}

\begin{proof}
 Set $\mu=\crit{j}$.   Given $\mu\leq\gamma<\kappa$, define $$U_\gamma ~ = ~  \Set{A\subseteq\mathcal{P}_\mu(\gamma)}{j[\gamma]\in j(A)}.$$
  Then it is easy to see that for every $\mu\leq\gamma<\kappa$, the collection $U_\gamma$ is a normal ultrafilter over $\mathcal{P}_\mu(\gamma)$. 
  Moreover, this definition directly ensures that $$U_\gamma ~ = ~  \Set{\Set{a\cap\gamma}{a\in A}}{A\in U_\delta}$$ holds for all $\mu\leq\gamma\leq\delta<\kappa$. 
  
  Now, given $\mu\leq\gamma<\kappa$, we let $i_\gamma :V \to M_\gamma$ denote the ultrapower embedding induced by $U_\gamma$. 
  In addition, for all $\mu\leq\gamma\leq\delta<\kappa$, we let $k_{\gamma,\delta}: M_\gamma \to M_\delta$ denote the canonical embedding satisfying $i_\delta=k_{\gamma,\delta}\circ i_\gamma$ (see {\cite[p. 333]{Kan:THI}}).

 \begin{claim*}
  If $\mu\leq\gamma<\kappa$ and $\gamma\leq\alpha<i_\gamma(\mu)$, then there exists $\gamma\leq\delta<\kappa$ with $k_{\gamma,\delta}(\alpha)=\delta$.  
 \end{claim*}
 
 \begin{proof}[Proof of the Claim]
  Pick a function $f:\mathcal{P}_\mu(\gamma)\to \mu$ with $[f]_{U_\gamma}=\alpha$ and define $$\delta ~ = ~ j(f)(j[\gamma]) ~ < ~ \kappa.$$  
  Since normality allows us to conclude  that $[a\mapsto{\rm ot}(a)]_{U_\gamma}=\gamma$, we know that $$\Set{a\in\mathcal{P}_\mu(\gamma)}{{\rm ot}(a)\leq f(a)} ~ \in ~  U_\gamma$$ and hence $$\gamma ~ = ~ {\rm ot}(j[\gamma]) ~ \leq ~ j(f)(j[\gamma]) ~ = ~  \delta.$$ 
  Moreover, we have $$j(f)(j(\gamma)\cap j[\delta]) ~ = ~  j(f)(j[\gamma]) ~ = ~  \delta  ~ = ~ {\rm ot}(j[\delta])$$ and hence $\Set{a\in\mathcal{P}_\mu(\gamma)}{f(a\cap\gamma)={\rm ot}(a)}\in U_\delta$. 
  But then 
  \begin{equation*}
   k_{\gamma,\delta}(\alpha) ~  = ~  k_{\gamma,\delta}([f]_{U_\gamma}) ~ = ~  [a\mapsto f(a\cap\gamma)]_{U_\delta} ~  = ~ [a\mapsto {\rm ot}(a)]_{U_\delta} ~ = ~    \delta. \qedhere
  \end{equation*}
 \end{proof}

  By {\cite[Theorem 24.11]{Kan:THI}}, this shows that $\mu$ is almost huge with target $\kappa$.
\end{proof}

\begin{corollary}\label{corollary:AlmostHugeFromPi1ic}
 Let $\kappa<\lambda$ be cardinals with the property that $\kappa$ is either parametrically $0$-exact for $\lambda$ or weakly parametrically $1$-exact for $\lambda$. Then the set of cardinals $\mu<\kappa$ with the property that $\mu$ is almost huge with target $\kappa$ is stationary in $\kappa$. 
 \end{corollary}
 
 \begin{proof}
  Let $C$ be a closed unbounded subset of $\kappa$. By definition, both of the listed assumption imply the existence of a non-trivial elementary embedding $j:V_\kappa\to V_\lambda$ with $j(\crit{j})=\kappa$ and $C\in\ran{j}$. %
  Then $\crit{j}$ is an element of $C$ and Lemma \ref{lemma:AlmostHugeFromEmb} shows that $\crit{j}$ is almost huge with target $\kappa$. 
 \end{proof}

 \begin{corollary}
 \label{corollary:almosthuge}
  Let $\kappa$ be a cardinal that is parametrically $0$-exact for some cardinal $\lambda>\kappa$. Then $\kappa$ is almost huge with target $\lambda$. 
 \end{corollary}
 
 \begin{proof}
  By definition, there exist a cardinal   $\kappa'>\kappa$, a cardinal  $\lambda<\lambda'\in C^{(1)}$, an elementary submodel $X$ of $H_{\kappa'}$ with $V_\kappa\cup\{\kappa\}\subseteq X$ and an elementary embedding $j:X\to H_{\lambda'}$ with $j(\kappa)=\lambda$ and $j(\crit{j})=\kappa$. 
  Then Lemma \ref{lemma:AlmostHugeFromEmb} implies that $\crit{j}$ is almost huge with target $\kappa$. Since the system of filters witnessing this statement is contained in $H_{\kappa'}$, the model  $X$ also contains such a system. 
  But then the elementarity of $j$ implies that, in $H_{\lambda'}$, there is a system of ultrafilters witnessing that $\kappa$ is almost huge with target $\lambda$. Since $\lambda<\lambda'\in C^{(1)}$, this statement  also holds in $V$.  
 \end{proof}

Recall that a cardinal $\kappa$ is \emph{huge} if there exists a transitive class $M$ and a non-trivial elementary embedding $j:V\to M$ with  $\crit{j}=\kappa$ and ${}^{j(\kappa)}M\subseteq M$.
We then say that a 
cardinal $\kappa$ is huge  with \emph{target $\lambda$} if there exists an embedding $j$ witnessing the hugeness of $\kappa$ with $j(\kappa)=\lambda$. 
It is well-known that $\kappa$ is huge with target $\lambda$ if and only if there exists a $\kappa$-complete normal ultrafilter $\mathcal{U}$ over $\mathcal{P}(\lambda)$ such that $\Set{ x\in \mathcal{P}(\lambda)} {\otp{x}=\kappa}\in \mathcal{U}$ (see {\cite[Theorem 24.8]{Kan:THI}}).

 \begin{proposition}\label{proposition:ESRfromHuge}
  If $\kappa$ is huge with target $\lambda$, then $\kappa$ is weakly parametrically $1$-exact for $\lambda$. 
\end{proposition}
 
 \begin{proof}
  Let $M$ be an inner model with ${}^\lambda M\subseteq M$ and let $j:V\to M$ be an elementary embedding with $\crit{j}=\kappa$ and $j(\kappa)=\lambda$.
  Fix $A\in V_{\lambda +1}$ and let $N$ be an elementary submodel of $H_{\lambda^+}$  of cardinality $\lambda$ with $V_\lambda\cup\{A,\lambda \}\subseteq N$. 
  We then have $N\in M$ and, since $H_{\lambda^+}=H_{\lambda^+}^M$,   $\Sigma_1$-absoluteness implies that $N$ is $\Pi_1(V_{\lambda +1})$-upwards correct in $M$. 
  Set $j_0=j\restriction N:N\to H_{j(\lambda)^+}^M$. 
  Then $j_0$ is an elementary embedding that is an element of $M$. 
   Thus, in $M$, there exists a transitive, $\Pi_1(V_{j(\kappa)+1})$-upwards correct set $K$ with $V_{j(\kappa)}\cup\{j(\kappa)\}\subseteq K$ (namely N) and an elementary embedding $k:K\to H_{j(\lambda^+)}$ with $k(\crit{k})=j(\kappa)$, $k(j(\kappa))=j(\lambda)$ and $j(A)\in\ran{k}$ (namely $j_0$). 
   Hence, the elementarity of $j$ implies that, in $V$, there exists a transitive, $\Pi_1(V_{\kappa+1})$-upwards correct set $K$ with $V_{\kappa}\cup\{\kappa\}\subseteq K$ and an elementary embedding $k:K\to H_{\lambda^+}$ with $k(\crit{k})=\kappa$, $k(\kappa)=\lambda$ and $A\in\ran{k}$.  
 \end{proof}


 This result also allows us to show that the consistency strength of huge cardinals is strictly larger than the consistency strength of weakly $1$-exact cardinals.

\begin{corollary}\label{corollary:HUGEstrongerthanPI1ESR}
  If $\kappa$ is huge with target $\lambda$, then there is an inaccessible cardinal $\rho<\kappa$ such that $\Pi_1(V_\rho)$-$\ESR(\rho,\kappa)$ holds in $V_\lambda$. 
\end{corollary}

\begin{proof}
Let  $j:V\to N$ be an elementary embedding with critical point $\kappa$ such that $j(\kappa)=\lambda$ and $N$ is closed under $\lambda$-sequences. Note that $H_{\lambda^+}\in V_{j(\lambda)}^N$. 
 Fix a $\Pi_1$-formula $\varphi(v_0,v_1)$ and $z\in V_\kappa$ with the property that, in $V_{j(\lambda)}^N$, the class $\Set{A}{\varphi(A,z)}$ consists of structures of the same type. 
 Pick a structure $B$ of rank $\lambda$ with the property that $\varphi(B,z)$ holds in $V_{j(\lambda)}^N$. Since $H_{\lambda^+}\subseteq V_{j(\lambda)}^N$, $\Sigma_1$-absoluteness implies that $\varphi(B,z)$ also holds in $V$. 
 By our assumptions, we can now apply Proposition \ref{proposition:ESRfromHuge} to find a structure $A$ of rank $\kappa$ with the property that $\varphi(A,z)$ holds in $V$ and an elementary embedding $i$ of $A$ into $B$. But then $A$ is contained in $V_{j(\lambda)}^N$, $\varphi(A,z)$  holds in $V_{j(\lambda)}^N$ and the map $i$ is an element of $V_{j(\lambda)}^N$. 
 
 These computations show that the principle $\Pi_1(V_\kappa)$-$\ESR(\kappa,\lambda)$ holds in $V_{j(\lambda)}^N$. Using the elementarity of $j$, we can now conclude that, in $V_\lambda$, there is an inaccessible cardinal $\rho<\kappa$ with the property that  $\Pi_1(V_\rho)$-$\ESR(\rho,\kappa)$ holds . 
\end{proof}

 We now show that the implication given by Proposition \ref{proposition:ESRfromHuge}  is optimal.

 \begin{proposition}\label{proposition:LeastHugeFailure}
  If $\kappa$ is the least huge cardinal, then $\kappa$ is not $1$-exact for any cardinal $\lambda>\kappa$. 
 \end{proposition}
 
 \begin{proof}
  Assume, towards a contradiction, that there exists a cardinal $\kappa<\kappa'\in C^{(1)}$,  cardinals $\lambda<\lambda'\in C^{(2)}$, an elementary submodel $X$ of $H_{\kappa'}$ with $V_\kappa\cup\{\kappa\}\subseteq X$ and an elementary embedding $j:X\to H_{\lambda'}$ with $j(\kappa)=\lambda$.  
  %
  Note that the statement ``\emph{There exists a huge cardinal smaller than $\lambda$}" can be formulated by a $\Sigma_2$-formula with parameter $\lambda$ and, since $\lambda<\lambda'\in C^{(2)}$, this statement holds in $H_{\lambda'}$. 
  But then the elementarity of $j$ and the fact that $\kappa'\in C^{(1)}$ allow us to conclude that there exists a huge cardinal smaller than $\kappa$, a contradiction. 
 \end{proof}

 The next proposition gives a consistency upper bound for the existence of a parametrically exact cardinal.
Recall that an \emph{$I 1$-embedding} is a non-trivial elementary embedding $j:V_{\delta +1}\to V_{\delta +1}$ for some limit ordinal $\delta$. Also, an \emph{$I2$-embedding} is an elementary embedding $j:V\to M$ for some  transitive class $M$ such that $V_\delta \subseteq M$ for some limit ordinal $\delta >\crit{j}$ satisfying $j(\delta)=\delta$.  Finally, an \emph{$I3$-embedding} is a non-trivial elementary embedding $j:V_{\delta}\to V_{\delta}$, for some limit ordinal $\delta$ (see \cite[\S 24]{Kan:THI}). 
 Note that, if $j:V_\delta\to V_\delta$ is an $I3$-embedding with critical point $\kappa$, then $V_\delta$ is a model of \rm{ZFC} and the sequence $\seq{j^m(\kappa)}{m<\omega}$  is cofinal in $\delta$.

 \begin{proposition}
 \label{proposition:upperbound}
  Assume that $\kappa$ is the critical point of an $I3$-embedding $j:V_\delta\to V_\delta$. If  $l,m,n<\omega$, then, in $V_\delta$, the  cardinal $j^l(\kappa)$ is parametrically $n$-exact for $j^{l+m+1}(\kappa)$.    
 \end{proposition}
 
 \begin{proof}
  Given $0<m<\omega$, set $\kappa_m=j^m(\kappa)$. Then, in $V_\delta$, every $\kappa_m$ is inaccessible and belongs to $C^{(n)}$, for all $n<\omega$. 
  Pick $0<m<\omega$, set $\lambda =\kappa_m$ and fix $A\in V_{\lambda +1}$. 
  Then, in $V_\delta$, the map $j^m\restriction H_{\kappa_{m+1}}:H_{\kappa_{m+1}}\to H_{\kappa_{2m+1}}$ witnesses that there exists an elementary embedding $i:H_{j^m(\kappa_1)}\to H_{j^m(\kappa_{m+1})}$ with $j^m(A)\in\ran{i}$, $i(j^m(\kappa))=j^m(\lambda)$ and $i(\crit{i})=j^m(\kappa)$.
  But then the elementarity of $j^m:V_\delta\to V_\delta$ implies that, in $V_\delta$, there exists an elementary embedding $i:H_{\kappa_1}\to H_{\kappa_{m+1}}$ with $A\in\ran{i}$, $i(\kappa)=\lambda$ and  $i(\crit{i})=\kappa$. 
  Since $\kappa<\kappa_1\in(C^{(n)})^{V_\delta}$ and $\lambda<\kappa_{m+1}\in (C^{(n)})^{V_\delta}$ for all $n<\omega$, these computations show that, in $V_\delta$, the cardinal  $\kappa$ is parametrically $n$-exact for $\lambda$ for all $n<\omega$. 
  %
  %
  By elementarity of the iterated embedding $j^l$, this yields the statement of the proposition.  
 \end{proof}

 In the following, we will derive a much lower  upper bound for the consistency strength of the existence of a cardinal $\kappa$ that is weakly parametrically $n$-exact for some cardinal $\lambda$ for all $n<\omega$.

\begin{definition}
 Given a natural number $n>0$,  a cardinal $\kappa$ is \emph{$n$-superstrong} if there exists a transitive class $M$ and an elementary embedding $j:V\to M$ with $\crit{j}=\kappa$ and  $V_{j^n(\kappa)}\subseteq M$.
 If, moreover, ${}^{j(\kappa)}V_{j^n(\kappa)}\subseteq M$, then we say that $\kappa$ is \emph{hugely $n$-superstrong}. 
\end{definition}

Notice that, given an elementary embedding $j:V\to M$ and a natural number $n>1$, the embedding $j$ witnesses that $\kappa$ is hugely $n$-superstrong if and only if it witnesses that $\kappa$ is $n$-superstrong and $\cof{j^n(\kappa)}>j(\kappa)$. 
Also note that every huge cardinal is hugely $1$-superstrong and, for $n>1$, every  almost $n$-huge cardinal\footnote{Recall that $\kappa$ is \emph{almost  $n$-huge} if there exists an elementary embedding $j:V\to M$, with $\crit{j}=\kappa$, and with $M$ transitive and closed under ${<}j^n(\kappa)$-sequences.} is hugely $n$-superstrong.

 \begin{proposition}\label{proposition:ESRfromHugely2Exact}
  If $\kappa$ is a hugely $2$-superstrong cardinal, then there exists an inaccessible cardinal $\lambda>\kappa$ and a cardinal $\rho>\lambda$ such that $V_\rho$ is a model of \rm{ZFC} and, in $V_\rho$, the cardinal $\kappa$ is weakly parametrically $n$-exact for $\lambda$ for all natural numbers $n>0$. 
 \end{proposition}
 
 \begin{proof}
 Let $j:V\to M$ with be an elementary embedding with $M$ transitive, $\crit{j}=\kappa$,   $V_{j^2(\kappa)}\subseteq M$, and $^{j(\kappa)}V_{j^2(\kappa)}\subseteq M$. Set $\lambda=j(\kappa)$ and $\rho=j^2(\kappa)$. Then our assumptions ensure that $\lambda$ is an inaccessible cardinal and $\rho$ is a cardinal with the property that $V_\rho$ is a model of \rm{ZFC}. 
  Notice that, since $V_\kappa \preceq V_{j(\kappa)}=V_\lambda$,  elementarity implies that  $V_\lambda=V_{j(\kappa)}\preceq V_{j(\lambda)} = V_\rho$ and therefore also $V_\rho =V_{j(\lambda)} \preceq V^M_{j(\rho)}$.
  
  Now, fix $A\in V_{\lambda+1}$ and, in $M$,  pick an elementary submodel $X$ of $V_\rho$ of cardinality $\lambda$ with $V_\lambda\cup\{A,\lambda\}\subseteq X$. Let $\pi:X\to N$ denote the corresponding transitive collapse.  Then $V_\lambda\cup\{A,\lambda\}\subseteq N$. %
 %
 Moreover, since $\pi\restriction(V_{\lambda+1}\cap X)=\id_{V_{\lambda+1}\cap X}$, it follows that  $N$ is $\Pi_n(V_{\lambda+1})$-upwards correct in $V_\rho$ for all $n<\omega$. But  since $V_\rho \preceq V^M_{j(\rho)}$,  the set $N$ is also $\Pi_n(V_{\lambda+1})$-upwards correct in $V^M_{j(\rho)}$ for all $n<\omega$. 
 Finally, pick a bijection $b:X\to\lambda$ with $b(\lambda)=0$ and $b(\gamma)=\omega\cdot(1+\gamma)$ for all $\gamma<\lambda$.  Set $$E ~ = ~ \Set{\langle b(x_0),b(x_1)\rangle}{x_0,x_1\in X, ~ x_0\in x_1} ~ \in ~ M.$$ Then the map $j\restriction\lambda:\langle\lambda,E\rangle\to\langle\rho,j(E)\rangle$ is an elementary embedding of $\calL_\in$-structures and, since it is a subset of $V_\rho$ of cardinality $\lambda$, the closure properties of $M$ ensure that this map is an element of $M$. 
  
  We now have that, in $M$, the map  $$i ~ = ~  j(b^{{-1}}) \circ({j\restriction\lambda})\circ b\circ\pi^{{-}1}:N\longrightarrow V^M_{j(\rho)}$$ is an elementary embedding with  $i\restriction\kappa=\id_\kappa$, $i(\kappa)=j(\kappa)=\lambda$,  $i(\lambda)=j(\lambda)=\rho$ and $j(A)\in\ran{i}$. 
   %
   %
   Finally, we know that $N\in V_\rho$, because $N$ is a subset of $V_\rho$ of cardinality $\lambda$ in $M$ and $\rho$ is inaccessible in $M$.  
    Since the closure properties of $M$ imply that $\rho$ is a limit cardinal of cofinality greater than $\lambda$ in $V$, 
   we can find an $M$-cardinal  $\rho <\eta<j(\rho)$ with $j(N)\subseteq H^M_\eta\prec  V^M_{j(\rho)}$. 
   
   Fix $0<n<\omega$. The above computations now show that, in $V_{j(\rho)}^M$, there exists a cardinal $\rho<\eta\in C^{(n)}$, a transitive $\Pi_n(V_{\lambda+1})$-upwards correct set $N$ with $V_\lambda\cup\{\lambda\}\subseteq N$ and a non-trivial elementary embedding $i:N\to H_\eta$ with $i(\crit{i})=\lambda$, $i(\lambda)=\rho$ and $j(A)\in\ran{i}$. 
   In this situation, the elementarity of $j$ implies that, in $V_\rho$, there exists a cardinal $\lambda<\lambda'\in C^{(n)}$, a transitive $\Pi_n(V_{\kappa+1})$-upwards correct set $N_0$ with $V_\kappa\cup\{\kappa\}\subseteq N_0$ and a non-trivial elementary embedding $i_0:N_0\to H_{\lambda'}$ with $i_0(\crit{i_0})=\kappa$, $i_0(\kappa)=\lambda$ and $A\in\ran{i_0}$. 
 \end{proof}


\section{Beyond huge reflection}\label{section:Beyond}

In this section, we introduce a generalization of the principle $\ESR_{\Ce}(\kappa,\lambda)$ to  sequences of cardinals in order to obtain principles of structural reflection that imply the existence of even stronger large cardinals. The following definition is motivated by the formulation of  \emph{Chang's Conjecture}.

\begin{definition}
 Let $0<\eta\leq\omega$  and  let $\calL$ be a first-order language containing unary predicate symbols $\vec{P}=\seq{\dot{P}_i}{i<\eta}$. 
 \begin{enumerate}
     \item Given a sequence $\vec{\mu}=\seq{\mu_i}{i<\eta}$ of cardinals with supremum $\mu$, an $\calL$-structure $A$ has \emph{type $\vec{\mu}$ (with respect to $\vec{P}$)} if the universe  of $A$ has rank $\mu$ and $\rank{\dot{P}_i^{A}}=\mu_i$ for all $i<\eta$. 
     
     \item Given a class $\Ce$ of $\calL$-structures and  a strictly increasing sequence $\vec{\lambda}=\seq{\lambda_i}{i<1+\eta}$ of cardinals, we let $\ESR_{\Ce}(\vec{\lambda})$ denote the statement that for every structure $B$ in $\Ce$ of type  $\seq{\lambda_{i+1}}{i<\eta}$, there exists an elementary embedding of a structure $A$ in $\Ce$ of type $\seq{\lambda_i}{i<\eta}$ into $B$. 
     
     \item Given a definability class $\Gamma$ and a class $P$, we let $\Gamma(P)$-$\ESR(\vec{\lambda})$ denote the statement that $\ESR_\Ce (\vec{\lambda})$ holds for every class $\Ce$ of structures of the same type that is $\Gamma$-definable with parameters in $P$.  
 \end{enumerate} 
\end{definition}

In order to determine the large cardinal strength of the above principles, we consider the following \emph{sequential} versions of $n$-exact and weakly $n$-exact cardinals:

\begin{definition}
 Let $0<\eta\leq\omega$ and let $\vec{\lambda}=\seq{\lambda_m}{m<\eta}$ be a strictly increasing sequence of  cardinals with supremum $\lambda$. 
 \begin{enumerate}
     \item Given $n<\omega$, a cardinal $\kappa<\lambda_0$ is \emph{$n$-exact for $\vec{\lambda}$}  if for every $A\in V_{\lambda +1}$, there exists a cardinal $\rho$,\footnote{Note that, in both parts of this definition, the listed requirements ensure that there is a unique cardinal $\rho$ with these  properties. If $\eta=1$, then $\lambda=\lambda_0$ and  $\rho=\kappa$. Next, if $1<\eta<\omega$, then $\lambda=\lambda_{\eta-1}$ and $\rho=\lambda_{\eta-2}$. Finally, if $\eta=\omega$, then $\lambda=\rho$.}  a cardinal $\kappa'\in C^{(n)}$ greater than $\beth_\rho$, a cardinal $\lambda' \in C^{(n+1)}$ greater than $\lambda$,  an  elementary submodel $X$ of $H_{\kappa'}$ with $V_{\rho}\cup \{\rho\}\subseteq X$, and 
 an elementary embedding $j:X\to H_{\lambda'}$ with $A\in\ran{j}$, $j(\rho)=\lambda$, $j(\kappa)=\lambda_0$ and $j(\lambda_{m-1})=\lambda_m$  for all $0<m<\eta$. 
   If we further require that $j(\crit{j})=\kappa$, then we say that \emph{$\kappa$ is parametrically $n$-exact for $\vec{\lambda}$}. 
   
   \item Given $0<n<\omega$, a cardinal $\kappa <\lambda_0$ is \emph{weakly $n$-exact for $\vec{\lambda}$} if for every $A\in V_{\lambda+1}$, there exists a cardinal $\rho$, a  transitive,   $\Pi_n(V_{\rho+1})$-upwards correct set $M$ with $V_\rho \cup \{\rho\}\subseteq M$, a cardinal $\lambda'\in C^{(n-1)}$ greater than $\beth_\lambda$    and an elementary embedding $j:M\to H_{\lambda'}$ with  $A\in\ran{j}$, $j(\rho)=\lambda$,  $j(\kappa)=\lambda_0$ and  $j(\lambda_{m-1})=\lambda_m$ for all $0<m<\eta$. 
   If we further require that $j(\crit{j})=\kappa$, then we say that $\kappa$ is \emph{weakly parametrically  $n$-exact for $\vec{\lambda}$}.
 \end{enumerate}
\end{definition}

Note that, if $n>0$ is a natural number and $\seq{\lambda_i}{i\leq n}$ is a strictly increasing sequence of cardinals such that  $\lambda_0$ is weakly $1$-exact for $\seq{\lambda_{i+1}}{i<n}$, then the instance $$(\lambda_n,\ldots,\lambda_1,\lambda_0) ~ \twoheadrightarrow ~ (\lambda_{n-1},\ldots,\lambda_0,{<}\lambda_0)$$ of Chang's Conjecture (see, for example, {\cite[p. 914]{MR2768692}}) holds true. 
 Analogous implications hold true for sequences of cardinals of length $\omega$.

We  show next that the large cardinal notions introduced above are located in the uppermost regions of the large cardinal hierarchy. 
 Recall that, given a natural number $n>0$, a cardinal $\kappa$ is \emph{$n$-huge}
 if there exists a transitive class $M$ and an elementary embedding $j:V\to M$ with $\crit{j}=\kappa$ and ${}^{j^n(\kappa)}M\subseteq M$.

\begin{proposition}
\label{proposition:SequentialWeaklyExactImplynHuge}
  Let $n>0$ be a natural number. 
 \begin{enumerate}
  \item If $\kappa$ is an $n$-huge cardinal, witnessed by an elementary embedding $j:V\to M$, then  $\kappa$ is weakly parametrically $1$-exact for the sequence $\seq{j^{m+1}(\kappa)}{m<n}$. 
  
  \item If $\kappa$ is a cardinal and $\vec{\lambda}=\seq{\lambda_m}{m\leq n}$ is a sequence of cardinals such that $\kappa$ is either weakly $1$-exact for $\vec{\lambda}$ or $0$-exact for $\vec{\lambda}$, then  some cardinal less than $\kappa$ is $n$-huge.
  \end{enumerate}
\end{proposition}

\begin{proof}
$(i)$ Set $\rho=j^{n-1}(\kappa)$ and $\lambda=j^n(\kappa)$. Fix $A\in V_{\lambda +1}$. Let $N$ be an elementary submodel of $H_{\lambda^+}$ of size $\lambda$  with $V_\lambda \cup \{ \lambda, A\}\subseteq N$.
 Then $N$ is an element of $M$ and $N$ is $\Pi_1(V_{\lambda+1})$-upwards correct in $M$. 
 Since the map $j\restriction N:N\to H_{j(\lambda)^+}^M$ is also contained in $M$, elementarity allows us to conclude that, in $V$, there exists a transitive, $\Pi_1(V_{\rho+1})$-upwards correct set $K$ with $V_\rho\cup\{\rho\}\subseteq K$ and an elementary embedding $k:K\to H_{\lambda^+}$ with $A\in\ran{k}$, $k(\rho)=\lambda$,  $k(\crit{k})=\kappa$ and $k(j^m(\kappa))=j^{m+1}(\kappa)$ for all $m<n$. 

  (ii) Set $\lambda=\lambda_n$ and $\rho=\lambda_{n-1}$.  
  Both of our assumptions then yield a cardinal $\lambda'>\lambda$, a set $X$ with $V_\rho\cup\{\rho\}\subseteq X$ and an elementary embedding $j:X\to H_{\lambda'}$ with  $\kappa\in\ran{j}$, $j(\kappa)=\lambda_0$ and $j(\lambda_{m-1})=\lambda_m$ for all $0<m\leq n$. 
  Set $\mu=\crit{j}<\kappa$. An easy induction then shows that $j^{m+2}(\mu)\leq\lambda_m$ holds for all $m\leq n$. 
  In particular, we have $j^n(\mu)<\rho$ and therefore $i=j\restriction V_\rho:V_\rho\to V_\lambda$ is an elementary embedding with $\crit{i}=\mu$ and $i^n(\mu)<\rho$. Using results of Kanamori (see {\cite[Theorem 24.8]{Kan:THI}}), we can now conclude that $\mu$ is $n$-huge. 
\end{proof}

\begin{proposition}
 Let $\vec{\lambda}=\seq{\lambda_m}{m<\omega}$ be a strictly increasing sequence of cardinals with supremum $\lambda$ and let $\kappa<\lambda_0$ be a cardinal.
 \begin{enumerate}
     \item If $\kappa$ is either   weakly $1$-exact  for $\vec{\lambda}$ or  $0$-exact for $\vec{\lambda}$, then there exists an  $I3$-embedding $j:V_\lambda\to V_\lambda$. 
     
     \item If $\kappa$ is either weakly parametrically  $1$-exact  for $\vec{\lambda}$ or parametrically $0$-exact for $\vec{\lambda}$, then the set of critical points of $I3$-embeddings is stationary in $\kappa$. 
 \end{enumerate}
\end{proposition}

\begin{proof}
  (i) Both of our assumptions ensure that there exists a cardinal $\lambda'>\lambda$, a set $X$ with $V_\lambda\cup\{\lambda\}\subseteq X$ and an elementary embedding $j:X\to H_{\lambda'}$ with   $j(\lambda)=\lambda$ and $j(\lambda_{m-1})=\lambda_m$ for all $0<m<\omega$. 
  Then $j\restriction V_\lambda$ is an  $I3$-embedding. 
  
  (ii)  Fix a closed unbounded subset $C$ of $\kappa$.
  By our assumptions, there exists a cardinal $\lambda'>\lambda$, a set $X$ with $V_\lambda\cup\{\lambda\}\subseteq X$ and an elementary embedding $j:X\to H_{\lambda'}$ with $C\in\ran{j}$, $j(\crit{j})=\kappa$, $j(\lambda)=\lambda$ and $j(\lambda_{m-1})=\lambda_m$ for all $0<m<\omega$. 
  Then $\crit{j}\in C$ and the map $j\restriction V_\lambda:V_\lambda\to V_\lambda$ is an $I3$-embedding with critical point $\crit{j}$. 
\end{proof}

\begin{proposition}\label{proposition:I1impliesSeqESR}
 If $\kappa$ is the critical point of an I1-embedding $j:V_{\lambda+1}\to V_{\lambda+1}$ and $k>0$ is a natural number, then $\kappa$ is weakly parametrically $1$-exact for $\seq{j^{k(m+1)}(\kappa)}{m<\omega}$. 
\end{proposition}

\begin{proof}
By standard coding arguments, our assumptions yields an elementary embedding $i:H_{\lambda^+}\to H_{\lambda^+}$ with $i\restriction V_\lambda=j\restriction V_\lambda$ and therefore $i\restriction V_\kappa=\id_{V_\kappa}$. 
  Fix $A\in V_{\lambda+1}$ and pick an elementary submodel $N$ of $H_{\lambda^+}$ of cardinality $\lambda$ with $\lambda\cup\{A,\lambda\}\subseteq N$. Then $\Sigma_1$-absoluteness implies that $N$ is $\Pi_1(V_{\lambda+1})$-upwards correct. 
  In this situation, the set $N$ and the map $i^k\restriction N:N\to i^k(N)$ witness that, in $H_{\lambda^+}$, there exists a transitive, $\Pi_1(V_{\lambda+1})$-upwards correct set $M$ with $V_\lambda\cup\{\lambda\}\subseteq M$ and an elementary embedding $l:M\to i^k(N)$ with $i^k(A)\in\ran{l}$, $l(\crit{l})=i^k(\kappa)$, $l(i^k(\kappa))=i^k(j^k(\kappa))$ and $l(i^k(j^{k m}(\kappa)))=i^k(j^{k(m+1)}(\kappa))$ for all $0<m<\omega$. 
  Using the elementarity of $i^k:V_{\lambda+1}\to V_{\lambda+1}$, $\Sigma_1$-absoluteness and the fact that $N$ is an elementary submodel of $H_{\lambda^+}$, we can now conclude that  there exists  a transitive, $\Pi_1(V_{\lambda+1})$-upwards correct set $M$ with $V_\lambda\cup\{\lambda\}\subseteq M$ and an elementary embedding $l:M\to H_{\lambda^+}$ with $A\in\ran{l}$, $l(\crit{l})=\kappa$, $l(\kappa)=j^k(\kappa)$ and $l(j^{k m}(\kappa))=j^{k(m+1)}(\kappa)$ for all $0<m<\omega$.
  %
\end{proof}

In the remainder of this section, we show how the validity of the principle $\ESR(\vec{\lambda})$ is connected to the existence of cardinals that are exact or weakly exact for certain sequences of cardinals.

\begin{lemma}\label{lemma:EmbCharESRPi1seq}
 The following statements are equivalent for every natural number $n>0$, every ordinal $0<\eta\leq\omega$, every strictly increasing sequence $\vec{\lambda}=\seq{\lambda_i}{i<1+\eta}$ of uncountable cardinals with supremum $\lambda$, and all $z\in V_{\lambda_0}$: 
 \begin{enumerate}
     \item $\Pi_n(\{z\})$-$\ESR_{\Ce}(\vec{\lambda})$.  
     
     \item For every $A\in V_{\lambda+1}$, there exists a cardinal $\rho$, a  transitive, $\Pi_n(V_{\rho+1})$-upwards correct set $M$ with $V_\rho\cup\{\rho\}\subseteq M$, a cardinal $\lambda<\lambda'\in C^{(n-1)}$      and an elementary embedding $j:M\to H_{\lambda'}$ such that $A\in\ran{j}$, $j(z)=z$, $j(\rho)=\lambda$ and  $j(\lambda_i)=\lambda_{i+1}$ for all $i<\eta$.   
     
     \item For all cardinals  $\lambda<\lambda'\in C^{(n)}$ and every $A\in V_{\lambda+1}$, there exists a cardinal $\rho$, a  transitive, $\Pi_n(V_{\rho+1})$-upwards correct set $M$ with $V_\rho\cup\{\rho\}\subseteq M$ and an elementary embedding $j:M\to H_{\lambda'}$ such that $A\in\ran{j}$, $j(z)=z$, $j(\rho)=\lambda$ and  $j(\lambda_i)=\lambda_{i+1}$ for all $i<\eta$.  
 \end{enumerate}
\end{lemma}

\begin{proof}
  First, assume that (ii) holds. 
 Fix a $\Pi_n$-formula $\varphi(v_0,v_1)$ with the property that $\Ce=\Set{A}{\varphi(A,z)}$ is a suitable class of structures and pick a structure $B$ in $\Ce$ of type  $\seq{\lambda_{i+1}}{i<\eta}$. 
 Since $B\in V_{\lambda +1}$, we can find a cardinal $\rho$, a  transitive, $\Pi_n(V_{\rho+1})$-upwards correct set $M$ with $V_\lambda\cup\{\lambda\}\subseteq M$, a cardinal $\lambda<\lambda'\in C^{(n-1)}$ and an elementary embedding $j:M\to H_{\lambda'}$ such that $A,\vec{\lambda}\in\ran{j}$, $j(z)=z$, $j(\rho)=\lambda$ and  $j(\lambda_i)=\lambda_{i+1}$ for all $i<\eta$. 
 Elementarity then implies that $\rho$ is the supremum of the sequence $\seq{\lambda_i}{i<\eta}$. 
  Moreover, the fact that $\lambda'\in C^{(n-1)}$  implies  that $\Pi_n$-statements are downwards absolute from $V$ to $H_{\lambda'}$, and this allows us to conclude  that $\varphi(B,z)$ holds in $H_{\lambda'}$. 
  Pick $A\in M\cap V_{\rho+1}$ with $j(A)=B$. 
  In this situation, the elementarity of $j$ and the fact that  $j(\lambda_i)=\lambda_{i+1}$ holds for all $i<\eta$ cause $\varphi(A,z)$ to hold in $M$ and, by the $\Pi_n(V_{\rho+1})$-upwards correctness of $M$, this shows that $A$ is a structure of type $\seq{\lambda_i}{i<\eta}$ in $\Ce$. 
  Finally, since $A$ has rank $\rho$ and $V_\rho$  is a subset of $M$, the map $j$ induces an elementary embedding of $A$ into $B$. This shows that (i) holds in this case.

  Next, assume that (i) holds and we shall prove (iii). So fix a cardinal  $\lambda<\lambda'\in C^{(n)}$ and  $A\in V_{\lambda+1}$. 
  Define $\calL$ to be the first-order language extending $\calL_\in$ by predicate symbols $\vec{P}=\seq{\dot{P}_i}{i<\eta}$ and constant symbols $\dot{A}$, $\dot{\lambda}$,  $\dot{z}$ and $\seq{\dot{\lambda}_i}{i<\eta}$.  
  Define $\Ce$ to be the class of all $\calL$-structures $\langle D,E,\vec{P},a,b,c,\vec{d}\rangle$ with the property that $E$ is a well-founded and extensional relation on $D$ and, if $\vec{P}=\seq{P_i}{i<\eta}$, $\vec{d}=\seq{d_i}{i<\eta}$ and  $\pi:\langle D,E\rangle\to\langle M,\in\rangle$ is the induced transitive collapse, then the following statements hold: 
  \begin{itemize}
      \item $\rank{D}$ is a cardinal and $V_{\rank{D}}\cup\{\rank{D}\}\subseteq M$. 
      
      \item $M$ is $\Pi_n(V_{\rank{D}+1})$-upwards correct. 
      
      \item $\pi(b)=\rank{D}$, $\pi(c)=z$ and $\pi(d_i)=\rank{P_i}$ for all $i<\eta$. 
      
      \item  $\seq{\rank{P_i}}{i<\eta}$ is a strictly increasing sequence of cardinals with supremum $\rank{D}$. 
  \end{itemize}
 Then $\Ce$ is definable by a $\Pi_n$-formula with parameter $z$. 
 
 Now, let $X$ be an elementary substructure of $V_{\lambda'}$ of cardinality $\beth_\lambda$ with $V_\lambda\cup\{A,\lambda\}\subseteq X$. 
 Pick a bijection $f:X\to V_\lambda$ 
 and let $R$ denote the induced binary relation on $V_\lambda$.  These choices ensure that the transitive collapse of $\langle V_\lambda,R\rangle$ is $\Pi_n(V_{\lambda+1})$-upwards correct and this allows us to conclude that $$\langle V_\lambda,R,\seq{\lambda_{i+1}}{i<\eta},f(A),f(\lambda),f(z),\seq{f(\lambda_{i+1})}{i<\eta}\rangle$$ is a structure in $\Ce$ of type $\seq{\lambda_{i+1}}{i<\eta}$. 
 By our assumptions, there exists an elementary embedding $i$ of a structure $$\langle D,E,\seq{P_i}{i<\eta},a,b,c,\seq{d_i}{i<\eta}\rangle$$ of type $\seq{\lambda_i}{i<\eta}$ in $\Ce$ into the above structure. 
 Let $\pi:\langle D,E\rangle\to\langle M,\in\rangle$ denote the corresponding transitive collapse and set $$j ~ = ~ f^{{-}1}\circ i\circ\pi^{{-}1}:M\to H_{\lambda'}.$$ 
 Set $\rho=\pi(b)=\rank{D}=\sup_{i<\eta}\lambda_i$. 
 Then $M$ is a $\Pi_n(V_{\rho+1})$-upwards correct set with $V_\rho\cup\{\rho\}\subseteq M$  and $j$ is an elementary embedding with  $j(\rho)=\lambda$, $j(z)=z$ and $A\in\ran{j}$. 
 Moreover, given $i<\eta$, we now have $\lambda_i=\rank{P_i}=\pi(d_i)$ and this allows us to conclude that $$j(\lambda_i) ~ = ~ (f^{{-}1}\circ i)(d_i) ~ = ~ \lambda_{i+1}.$$ This shows that (iii) holds in this case. 
 
 Since (iii) obviously implies (ii), this concludes the proof of the lemma.  
\end{proof}

\begin{corollary}\label{corollary:SeqWeaklyExactYieldsReflcetion}
 Let $0<n<\omega$, let $0<\eta\leq\omega$ and let $\vec{\lambda}=\seq{\lambda_i}{i<1+\eta}$ be a strictly increasing sequence of cardinals. 
  \begin{enumerate}
   \item The cardinal $\lambda_0$ is weakly $n$-exact for $\seq{\lambda_{i+1}}{i<\eta}$ if and only if  $\Pi_n$-$\ESR(\vec{\lambda})$ holds. 
   
   \item If $\lambda_0$ is weakly parametrically $n$-exact for $\seq{\lambda_{i+1}}{i<\eta}$, then $\Pi_n(V_{\lambda_0})$-$\ESR(\vec{\lambda})$ holds. \qed 
  \end{enumerate}
\end{corollary}

 In the case of sequences of finite length, we can now generalize Theorem \ref{theorem:CharacterizationPI-ESR} to principles of the form  $\Pi_n$-$\ESR(\vec{\lambda})$.

\begin{theorem}\label{theorem:SeqEquivalenceWeaklyExactReflection}
 The following statements are equivalent for every cardinal $\kappa$ and all natural numbers $n,\eta>0$: 
  \begin{enumerate}
      \item $\kappa$ is the least cardinal such that there exists a strictly increasing sequence $\bar{\lambda}=\seq{\lambda_i}{i<\eta+1}$ of cardinals with $\lambda_0=\kappa$ and the property that $\Pi_n$-$\ESR(\vec{\lambda})$ holds. 

      \item $\kappa$ is the least cardinal such that there exists a strictly increasing sequence $\bar{\lambda}=\seq{\lambda_i}{i<\eta+1}$ of cardinals with $\lambda_0=\kappa$ and the property that $\Pi_n(V_\kappa)$-$\ESR(\vec{\lambda})$ holds. 
      
     \item $\kappa$ is the least cardinal that is weakly $n$-exact for some strictly increasing sequence of cardinals greater than $\kappa$ of length $\eta$. 
      
    \item $\kappa$ is the least cardinal that is weakly parametrically $n$-exact for some strictly increasing sequence of cardinals greater than $\kappa$ of length $\eta$.
  \end{enumerate}
\end{theorem}

\begin{proof}
 Let $\kappa$ be the least cardinal such that there exists a strictly increasing sequence $\vec{\lambda}=\seq{\lambda_i}{i<\eta+1}$ of cardinals with $\lambda_0=\kappa$ and the property that $\Pi_n$-$\ESR(\vec{\lambda})$ holds. Set $\lambda=\lambda_\eta$ and $\rho=\lambda_{\eta-1}$. Pick $\lambda'>\lambda$ with the property that $V_{\lambda'}$ is sufficiently elementary in $V$. 
  
  \begin{claim*}
   $\kappa$ is weakly parametrically $n$-exact for $\seq{\lambda_{i+1}}{i<\eta}$. 
  \end{claim*}
  
  \begin{proof}[Proof of the Claim]
   Assume, towards a contradiction, that $A\in V_{\lambda+1}$ witnesses   that $\kappa$ is not weakly parametrically $n$-exact for $\seq{\lambda_{i+1}}{i<\eta}$.
   Using Lemma \ref{lemma:EmbCharESRPi1seq}, we can find a transitive, $\Pi_n(V_{\rho+1})$-upwards correct set $M$ with $V_\rho\cup\{\rho\}\subseteq M$ and an elementary embedding $j:M\to V_{\lambda'}$ with  $A,\kappa\in\ran{j}$  and   $j(\lambda_i)=\lambda_{i+1}$ for all $i<\eta$. 
   Our setup then ensures that $j(\crit{j})\neq\kappa$ and, since $\kappa\in\ran{j}$ and $j(\kappa)=\lambda_1>\kappa$, this implies that $j(\crit{j})<\kappa$. 
   Given $i<\eta+2$, set $\mu_i=j^i(\crit{j})$. Then $\mu_{i+1}\leq\lambda_i$ for all $i<\eta+1$. 
   Since $\mu_1<\kappa$, the minimality of $\kappa$ yields a $\Pi_n$-formula $\varphi(v)$ with the property that the class $\Ce=\Set{A}{\varphi(A)}$ consists of structures of the same type and there exists a structure $B$ of type $\seq{\mu_{i+2}}{i<\eta}$ in $\Ce$ such that for every structure $A$ of type $\seq{\mu_{i+1}}{i<\eta}$ in $\Ce$, there exists no elementary embedding of $A$ into $B$. 
   By the choice of $\lambda'$, these statements are absolute between $V$ and $V_{\lambda'}$. 
   Elementarity now implies that,  in $M$, there exists a structure $B_0$ of type $\seq{\mu_{i+1}}{i<\eta}$ with the property that $\varphi(B_0)$ holds and for every structure $A$ of type $\seq{\mu_i}{i<\eta}$ such that $\varphi(A)$ holds, there is no elementary embedding of $A$ into $B_0$. 
   Since $M$ is $\Pi_n(V_{\rho+1})$-upwards correct and $V_{\lambda'}$ is sufficiently elementary in $V$, it follows that both $\varphi(B_0)$ and $\varphi(j(B_0))$ hold in $V_{\lambda'}$. But we can now use the elementarity of $j$ to derive a contradiction, because  $j(B_0)$ is a structure of type $\seq{\mu_{i+2}}{i<\eta}$, $B_0$ is a structure of type  $\seq{\mu_{i+1}}{i<\eta}$ and the map $j$ induces an elementary embedding of $B_0$ into $j(B_0)$ that is an element of $V_{\lambda'}$. 
  \end{proof}
  
  With the help of the above claim, we can apply Lemma \ref{lemma:EmbCharESRPi1seq} and Corollary \ref{corollary:SeqWeaklyExactYieldsReflcetion} to conclude that all statements listed in the theorem are equivalent. 
\end{proof}

 Note that the above proof cannot be directly generalized to sequences of cardinals of length $\omega$, because, if $\kappa$ is the least cardinal with the property that $\Pi_n$-$\ESR(\vec{\lambda})$ holds for some strictly increasing sequence $\vec{\lambda}=\seq{\lambda_i}{i<\omega}$ of cardinals with $\lambda_0=\kappa$ and supremum $\lambda$, we assume that $\kappa$
 is not weakly parametrically $n$-exact for $\seq{\lambda_{i+1}}{i<\omega}$ and we repeat the above construction to obtain a  transitive, $\Pi_n(V_{\lambda+1})$-upwards correct set $M$ with $V_\lambda\cup\{\lambda\}\subseteq M$ and an elementary embedding $j:M\to V_{\lambda'}$ with   $j(\crit{j})<\kappa$ and   $j(\lambda_i)=\lambda_{i+1}$ for all $i<\omega$, then we do not know wheter the sequence $\seq{j^{i+1}(\crit{j})}{i<\omega}$ is contained in the range of $j$ and this stops us from repeating the above minimality argument.

 Analogously to the above results, the statement of Theorem \ref{theorem:CharacterizationSigma-ESR} can be generalized to the context of this section.

\begin{lemma}\label{lemma:SeqEquivESR}
 The following statements are equivalent for every natural number $n>0$, every ordinal $0<\eta\leq\omega$, every strictly increasing sequence $\vec{\lambda}=\seq{\lambda_i}{i<1+\eta}$ of uncountable cardinals with supremum  $\lambda$, and all $z\in V_{\lambda_0}$: 
 \begin{enumerate}
     \item $\Sigma_{n+1}(\{z\})$-$\ESR_{\Ce}(\vec{\lambda})$.  
     
     \item For every $A\in V_{\lambda +1}$, there exists a cardinal $\rho$, a cardinal $\kappa'\in C^{(n)}$ greater than $\rho$, a cardinal $\lambda' \in C^{(n+1)}$ greater than $\lambda$,  an  elementary submodel $X$ of $H_{\kappa'}$ with $V_{\rho}\cup \{\rho\}\subseteq X$, and 
 an elementary embedding $j:X\to H_{\lambda'}$ with $A\in\ran{j}$, $j(z)=z$,  $j(\rho)=\lambda$  and $j(\lambda_i)=\lambda_{i+1}$  for all $i<\eta$.   
 \end{enumerate}
\end{lemma}

\begin{proof}
 Assume that (ii) holds. Pick a $\Sigma_{n+1}$-formula $\varphi(v_0,v_1)$ with the property that the class $\Ce=\Set{A}{\varphi(A,z)}$ consists of suitable structures and fix a structure $B$ in $\Ce$ of type $\seq{\lambda_{i+1}}{i<\eta}$. 
 Since $B\in V_{\lambda +1}$, there exists a cardinal $\rho$, a cardinal  $\kappa'$ with $\rho<\kappa'\in C^{(n)}$, a cardinal $\lambda'$ with   $\lambda<\lambda'\in C^{(n+1)}$, an  elementary submodel $X$ of $V_{\kappa'}$ with $V_{\rho}\cup \{\rho \}\subseteq X$, and 
 an elementary embedding $j:X\to H_{\lambda'}$ with $B\in\ran{j}$, $j(z)=z$, $j(\rho)=\lambda$ and  $j(\lambda_i)=\lambda_{i+1}$  for all $i<\eta$. 
 Pick  $A\in X$ such that  $j(A)=B$. Since $\lambda' \in C^{(n+1)}$, we know that  $\varphi(B,z)$ holds in $V_{\lambda'}$. The elementarity of $j$ then implies that  $\varphi(A,z)$ holds in  $V_{\kappa'}$ and,  since $\kappa'\in C^{(n)}$,  we may then    conclude that $A\in \Ce$.
  Moreover, since  $j(\lambda_i)=\lambda_{i+1}$ holds for all $i<\eta$,   elementarity allows us to  conclude that $A$ has type $\seq{\lambda_i}{i<\eta}$. 
  Finally, elementarity also implies that $A$ has rank $\rho$ and, since $V_\rho$  is a subset of $X$, the map $j$ induces an elementary embedding of $A$ into $B$.

 Next, assume that (i) holds and fix $A\in V_{\lambda+1}$. 
 Let $\calL$  be the first-order language extending $\calL_\in$ by  predicate symbols $\vec{P}=\seq{\dot{P}_i}{i<\eta}$ and constant symbols $\dot{A}$, $\dot{\lambda}$, $\dot{z}$ and 
 $\seq{\dot{\lambda}_i}{i<\eta}$.  
 Define $\Ce$ to be the class of all $\calL$-structures $\langle D,E,\vec{P},a,b,c,\vec{d}\rangle$ for which  there exists $\theta \in C^{(n)}$ and an isomorphism  $\tau$ between $\langle D,E\rangle$ and an elementary substructure $X$ of $H_\theta$ with the property that, if $\vec{P}=\seq{P_i}{i<\eta}$ and $\vec{d}=\seq{d_i}{i<\eta}$,  then the following statements   hold: 
  \begin{itemize}
        \item $\rank{D}$ is a cardinal and $V_{\rank{D}}\cup\{\rank{D}\}\subseteq X$. 
  
      \item $\tau(b)=\rank{D}$, $\tau(c)=z$ and $\tau(d_i)=\rank{P_i}$ for all $i<\eta$.   
  
      \item $\seq{\rank{P_i}}{i<\eta}$ is a strictly increasing sequence of cardinals with supremum $\rank{D}$. 
      
  \end{itemize}
 Then  the class $\Ce$ is then  $\Sigma_{n+1}$-definable  with parameter $z$.

  Pick $\lambda<\lambda'\in C^{(n+1)}$,  an elementary substructure $Y$ of $H_{\lambda'}$ of cardinality $\beth_\lambda$ with    $V_\lambda\cup\{A,\lambda\}\subseteq Y$ and a bijection $f:Y\to V_\lambda$. 
  %
  %
  If we now let $R$ denote the binary relation on $V_\lambda$ induced by $f$, then the resulting $\calL$-structure $$\langle V_\lambda,R,\seq{\lambda_{i+1}}{i<\omega},f(A),f(\lambda),f(z),\seq{f(\lambda_{i+1})}{i<\eta}\rangle$$ is an element of $\Ce$ of type $\seq{\lambda_{i+1}}{i<\eta}$. 
  Set $\rho=\sup_{i<\eta}\lambda_i$. 
   By our assumptions, there exists an elementary embedding $i$ of a structure $\langle D,E,\vec{P},a,b,c,\vec{d}\rangle$ of type $\seq{\lambda_i}{i<\eta}$ in $\Ce$ into the above structure. 
  Pick a cardinal $\kappa'\in C^{(n)}$, an elementary submodel $X$ of $H_{\kappa'}$ and an isomorphism $\tau:\langle D,E\rangle\to \langle X,\in\rangle$ witnessing that the given structure is contained in $\Ce$. 
  Then $\rho$ is a cardinal with $\tau(b)=\rho=\rank{D}$ and $V_\rho\cup\{\rho\}\subseteq X$. Moreover, we have $\tau(c)=z$ and  $\lambda_i=\rank{P_i}=\tau(d_i)$ for all $i<\eta$.  
  If we now define  $$j ~ = ~ f^{{-}1}\circ i\circ\tau^{{-}1}:X\to H_{\lambda'},$$ then  $j$ is an elementary embedding with $A\in\ran{j}$, $j(z)=z$, $j(\rho)=\lambda$ and $j(\lambda_i)=\lambda_{i+1}$ for all $i<\eta$. 
\end{proof}

\begin{corollary}\label{corollary:SeqEquivSigma}
 Let $0<n<\omega$, let $0<\eta\leq\omega$ and let $\vec{\lambda}=\seq{\lambda_i}{i<1+\eta}$ be a strictly increasing sequence of cardinals. 
  \begin{enumerate}
   \item The cardinal $\lambda_0$ is  $n$-exact for $\seq{\lambda_{i+1}}{i<\eta}$ if and only if  $\Sigma_{n+1}$-$\ESR(\vec{\lambda})$ holds. 
   
   \item If $\lambda_0$ is  parametrically $n$-exact for $\seq{\lambda_{i+1}}{i<\eta}$, then $\Sigma_{n+1}(V_{\lambda_0})$-$\ESR(\vec{\lambda})$ holds. \qed 
  \end{enumerate}
\end{corollary}

As above, we can now generalize Theorem \ref{theorem:CharacterizationSigma-ESR} to the principle $\Sigma_{n+1}$-$\ESR(\vec{\lambda})$ for finite sequences of cardinals $\vec{\lambda}$.

\begin{theorem} 
\label{theorem:SeqEquivalenceExactReflection}
 The following statements are equivalent for every cardinal $\kappa$ and all natural numbers $n,\eta>0$: 
  \begin{enumerate}
      \item $\kappa$ is the least cardinal such that there exists a strictly increasing sequence $\bar{\lambda}=\seq{\lambda_i}{i<\eta+1}$ of cardinals with $\lambda_0=\kappa$ and the property that $\Sigma_{n+1}$-$\ESR(\vec{\lambda})$ holds. 

      \item $\kappa$ is the least cardinal such that there exists a strictly increasing sequence $\bar{\lambda}=\seq{\lambda_i}{i<\eta+1}$ of cardinals with $\lambda_0=\kappa$ and the property that $\Sigma_{n+1}(V_\kappa)$-$\ESR(\vec{\lambda})$ holds. 
      
     \item $\kappa$ is the least cardinal that is  $n$-exact for some strictly increasing sequence of cardinals greater than $\kappa$ of length $\eta$. 
      
    \item $\kappa$ is the least cardinal that is  parametrically $n$-exact for some strictly increasing sequence of cardinals greater than $\kappa$ of length $\eta$.
  \end{enumerate}
\end{theorem}

\begin{proof}
 Assume that  $\kappa$ is the least cardinal such that $\Sigma_{n+1}$-$\ESR(\vec{\lambda})$ holds for some   strictly increasing sequence $\vec{\lambda}=\seq{\lambda_i}{i<\eta+1}$ of cardinals with $\lambda_0=\kappa$.  
 Set $\lambda=\lambda_\eta$ and $\rho=\lambda_{\eta-1}$. Pick $\lambda'>\lambda$ with the property that $V_{\lambda'}$ is sufficiently elementary in $V$. 
  
  \begin{claim*}
   $\kappa$ is  parametrically $n$-exact for $\seq{\lambda_{i+1}}{i<\eta}$. 
  \end{claim*}
  
  \begin{proof}[Proof of the Claim]
   Assume, towards a contradiction, that $A\in V_{\lambda+1}$ witnesses that $\kappa$ is not parametrically $n$-exact for $\seq{\lambda_{i+1}}{i<\eta}$. 
   By Lemma \ref{lemma:SeqEquivESR},  a cardinal $\kappa'$ with   $\rho<\kappa'\in C^{(n)}$, a cardinal $\lambda'$ with $\lambda<\lambda'\in C^{(n+1)}$, an elementary submodel $X$ of $H_{\kappa'}$ with $V_\rho\cup\{\rho\}\subseteq X$ and an elementary embedding $j:X\to H_{\lambda'}$ with $A,\kappa\in\ran{j}$    and $j(\lambda_i)=\lambda_{i+1}$ for all $i<\eta$. 
   We then know that $j(\crit{j})\neq\kappa$ and, since $\kappa\in\ran{j}$ and $j(\kappa)>\kappa$, this allows us to conclude  $j(\crit{j})<\kappa$. 
   Given $i<\eta+2$, set $\mu_i=j^i(\crit{j})$.  
   Since $\mu_1<\kappa$, the minimality of $\kappa$ yields a $\Sigma_{n+1}$-formula $\varphi(v)$ with the property that the class $\Ce=\Set{A}{\varphi(A)}$ consists of structures of the same type and there exists a structure $B$ of type $\seq{\mu_{i+2}}{i<\eta}$ in $\Ce$ such that for every structure $A$ of type $\seq{\mu_{i+1}}{i<\eta}$ in $\Ce$, there exists no elementary embedding of $A$ into $B$. 
   Our set up now ensures that,  in $X$, there exists a structure $B_0$ of type $\seq{\mu_{i+1}}{i<\eta}$ with the property that $\varphi(B_0)$ holds and for every structure $A$ of type $\seq{\mu_i}{i<\eta}$ such that $\varphi(A)$ holds, there is no elementary embedding of $A$ into $B_0$. 
   Then both $\varphi(B_0)$ and $\varphi(j(B_0))$ hold in $V_{\lambda'}$, and $j$ induces an elementary embedding of $B_0$ into $j(B_0)$ that is an element of $V_{\lambda'}$. 
   Since $B_0$ has type $\seq{\mu_{i+1}}{i<\eta}=j(\seq{\mu_i}{i<\eta})$ and $j(B_0)$ has type $\seq{\mu_{i+2}}{i<\eta}=j(\seq{\mu_{i+1}}{i<\eta})$, we can now use the elementarity of $j$ to derive a contradiction.  
  \end{proof}
  
  A combination  of this claim and  Corollary \ref{corollary:SeqEquivSigma} now yields the desired equivalences. 
\end{proof}


\section{Open questions and concluding remarks}

We close this paper by discussing some questions raised by the above results.

 First, recall that a cardinal $\kappa$ is \emph{superhuge} if there is a proper class of cardinals $\lambda$ with the property that $\kappa$ is huge with target $\lambda$. 
 In order to study principles of structural reflection related to superhugeness,  Proposition \ref{proposition:ESRfromHuge} suggests to study cardinals $\kappa$ that are weakly $1$-exact for a proper class of cardinals $\lambda$. 
 By Corollary \ref{corollary:WeakExactStronger}, this property is equivalent to the assumption that  the principle $\Pi_1$-$\ESR(\kappa,\lambda)$ holds for a proper class of cardinals $\lambda$. 
 Therefore, it is natural to ask whether a variation of Theorem \ref{theorem:CharacterizationPI-ESR} can be proven for these cardinals.

\begin{question}
 Are the following statements equivalent for every cardinal $\kappa$ and every natural number $n>0$? 
\begin{enumerate}
 \item $\kappa$ is the least cardinal that is weakly parametrically $n$-exact for a proper class of cardinals $\lambda$. 
 
 
  \item $\kappa$ is the least cardinal with the property that $\Pi_n$-$\ESR(\kappa,\lambda)$ holds for a proper class of cardinals $\lambda$. 
  
\end{enumerate}
\end{question}

Our next question deals with the exact position of $n$-exact and weakly $n$-exact cardinals in the large cardinal hierarchy. 
 By Corollary \ref{corollary:AlmostHugeFromPi1ic}  and Proposition \ref{proposition:ESRfromHugely2Exact}, these notions are properly contained  in the interval given by \emph{almost hugeness} and \emph{almost $2$-hugeness}. Moreover, Corollary \ref{corollary:HUGEstrongerthanPI1ESR} shows that hugeness is strictly stronger than weak $1$-exactness. Finally, Proposition \ref{proposition:LeastHugeFailure} implies that, if $\kappa$ is the least huge cardinal, then $\Sigma_2$-$\ESR(\kappa,\lambda)$ fails for all $\lambda>\kappa$.   
 These results leave open the precise relationship between hugeness and exactness, and  motivate the following question:

\begin{question}
 Does the consistency of the theory $\textrm{ZFC}+``\textit{there exists a huge cardinal}"$ imply the consistency of the theory $\mathrm{ZFC}+ ``\textit{$\Sigma_2$-$\ESR(\kappa)$ holds for some  cardinal  $\kappa$}"$? 
\end{question}

The results of {\cite[Section 2.2.1]{MR2538021}} might provide tools to derive a negative answer to this question.

We finally discuss some questions left open about the infinite sequential versions of exact structural reflection principles introduced in Section \ref{section:Beyond}. In the light of Proposition \ref{proposition:I1impliesSeqESR}, it is natural to ask whether the consistency of principles of the form $\Sigma_n$-$\ESR(\vec{\lambda})$ for infinite sequences $\vec{\lambda}$ of cardinals and natural numbers $n>1$ can be established from some very strong large cardinal assumption (like $\mathrm{ZFC}+"\textit{There exists an $I0$-cardinal\hspace{1.7pt}}"$), or whether these principles are outright inconsistent with \rm{ZFC}.

\begin{question}
 Does \rm{ZFC} prove that the principle $\Sigma_2$-$\ESR(\vec{\lambda})$ fails for every strictly increasing sequence $\vec{\lambda}$ of cardinals of length $\omega$? 
\end{question}

However, if we only assume \rm{ZF} and  $\kappa$ is a Reinhardt cardinal, witnessed by an elementary embedding $j:V\to V$, then for every  natural number $n>0$, the critical  sequence $\vec{\lambda}=\seq{\lambda_i}{i<\omega}$, given by  $\lambda_i=j^i(\crit{j})$, witnesses that the principle  $\Pi_n(V_{\lambda_0})$-$\ESR(\vec{\lambda})$ holds. 

\medskip

Finally, as noted in the discussion following Theorem \ref{theorem:SeqEquivalenceWeaklyExactReflection}, our techniques do not allow us to generalize Theorem \ref{theorem:CharacterizationPI-ESR} to principles of the form $\ESR(\vec{\lambda})$ for infinite sequences $\vec{\lambda}$. This motivates the following question:

\begin{question}
 If $\kappa$ is the least cardinal with the property that $\Pi_n$-$\ESR(\vec{\lambda})$ holds for some strictly increasing sequence $\vec{\lambda}$ of cardinals of length $\omega$ with minimum $\kappa$, is $\kappa$ weakly parametrically $n$-exact for some strictly increasing sequence of cardinals greater than $\kappa$ of length $\omega$? 
\end{question}


\bibliographystyle{plain} 
\bibliography{masterbiblio}

\end{document}